\documentclass[a4paper,11pt]{amsart}

\usepackage{amscd}
\usepackage[arrow, matrix, curve]{xy}
\usepackage{amssymb}

\usepackage{fullpage}
\usepackage{color}
\usepackage{euscript}
\usepackage{tikz}
\newlength{\defbaselineskip} 
\usepackage{footnote}
\usepackage{color}
\usepackage{todonotes}
\usepackage{wrapfig}

\usepackage[bookmarks, colorlinks, breaklinks, pdftitle={On the Morin problem I},
pdfauthor={G.Kapustka, A.Verra}]{hyperref}
\hypersetup{linkcolor=blue,citecolor=blue,filecolor=black,urlcolor=blue}

\newtheorem{thm}{Theorem}[section]

\newtheorem{lemm}[thm]{Lemma}

\newtheorem{prop}[thm]{Proposition}

\newtheorem{prob}[thm]{Problem}

\theoremstyle{definition}
\newtheorem{defi}[thm]{Definition}

\newtheorem{rem}[thm]{Remark}
\usetikzlibrary{matrix,arrows}
\makeatletter
\tikzset{
  edge node/.code={%
      \expandafter\def\expandafter\tikz@tonodes\expandafter{\tikz@tonodes #1}}}
\makeatother
\tikzset{
  subseteq/.style={
    draw=none,
    edge node={node [sloped, allow upside down, auto=false]{$\subseteq$}}},
  Subseteq/.style={
    draw=none,
    every to/.append style={
      edge node={node [sloped, allow upside down, auto=false]{$\subseteq$}}}
  }
}

 \numberwithin{equation}{section}
\numberwithin{equation}{section} \theoremstyle{definition}
\DeclareMathOperator{\Aut}{Aut}

 \DeclareMathOperator{\Spec}{Spec}

\DeclareMathOperator{\mi}{min}

\DeclareMathOperator{\Sym}{Sym}

\DeclareMathOperator{\di}{div}

 \DeclareMathOperator{\Ker}{Ker}

\DeclareMathOperator{\Sing}{Sing}

\DeclareMathOperator{\CH}{CH}

          \newcommand\PP{{\mathbb{P}}}

\definecolor{zielony}{rgb}{0.5, 0.9, 0.1}
\definecolor{czerwony}{rgb}{0.8, 0.2, 0.1}
\definecolor{niebieski}{rgb}{0.3, 0.1, 0.9}

\newcounter{appendice}

\begin{document}
\title{\small On Morin configurations of higher length}

\author[G.~Kapustka]{Grzegorz Kapustka}\thanks{GK partially supported by Narodowe Centrum Nauki 2018/30/E/ST1/00530}
\address{Department of Mathematics and Informatics,
Jagiellonian University, {\L}ojasiewicza 6, 30-348 Krak{\'o}w, Poland.}
\email{grzegorz.kapustka@uj.edu.pl}
\author[A.~Verra ]{Alessandro Verra} \thanks{AV partially supported by INdAM-GNSAGA and by PRIN-2015 'Geometry of Algebraic Varieties'}
\address{Universita  Roma Tre, Dipartimento di Matematica e Fisica, Largo San Leonardo Murialdo 1 \ 00146 Roma, Italy}
\email{ verra@mat.unirom3.it}
\maketitle
\begin{abstract}
This paper studies finite Morin configurations $F$ of planes in $\mathbb P^5$ having higher length. The uniqueness of the configuration of maximal cardinality $20$
is proven. This is related to the stable canonical genus $6$ curve $C_{\ell}$ union of the $10$ lines of a smooth quintic Del Pezzo surface $Y$ in $\mathbb P^5$ and to the Petersen graph.
Families of length $\geq 16$, previously unknown, are constructed by smoothing partially $C_{\ell}$. A more general irreducible family of special configurations of 
length $\geq 11$, we name as Morin-Del Pezzo configurations,  is considered and studied. This depends on $9$ moduli and is defined via the family of nodal and rational
canonical curves of $Y$. The special relations between Morin-Del Pezzo configurations and the geometry of special threefolds, like the 
Igusa quartic or its dual Segre primal, are focused.
 \end{abstract}

\section{Introduction}
Let $\mathbb G$ be the Grassmannian of $n$-spaces of $\mathbb P^{2n+c}$,  $c \geq 1$. For any $u \in \mathbb G$ we denote by $P_u$ its corresponding $n$-space and by $\sigma_u$ the codimension $c$  Schubert variety \begin{equation} \sigma_u := \lbrace e \in \mathbb G \ \vert \ P_u \cap P_e \neq \emptyset \rbrace. \end{equation}  
A \it scheme of incident $n$-spaces \rm is a closed scheme $F \subset \mathbb G$ satisfying the condition
 \begin{equation}
F \subseteq \bigcap_{u \in F} \sigma_u.
\end{equation}
This  implies $P_u \cap P_v \neq \emptyset$, $\forall u,v \in F$. We say that $F$ is \it complete \rm if the equality holds.     \begin{defi} \it A Morin configuration is a complete scheme $F$ of incident $n$-spaces. \end{defi} 
Integral Morin configurations $F$ of planes in $\mathbb P^5$ were classified in 1930 by Morin himself if $\dim F > 0$,  see \cite{Morin}. In the same paper the following problem is posed:
\begin{prob}\label{Mor} Classify finite Morin configurations  of planes in $\PP^5$.
\end{prob}
 Notice that, as Zak points out, the analogous classification in $\mathbb P^{2n+c}$ is elementary in the case $2n + c + 1 \neq \binom{n+2}2$, see \cite{A}. 
 Morin problem  in $\mathbb P^5$, which is specially related  to hyperk\"ahler geometry, was readdressed in \cite{DM} by Dolgachev and Markushevich. They construct and study configurations of minimal cardinality $10$ and their families. In \cite{O} O' Grady proved the existence of configurations of cardinality $k$ for any $10\leq k\leq 16$. Next he showed that a finite Morin configuration of planes in $\mathbb P^5$
has length $k \leq 20$ and asked about the missing cases. The main result of \cite{DGKKW} is the construction of a finite Morin configuration of planes in $\PP^5$ of cardinality $20$.    In this paper we contribute to Morin problem and to describe the geometry of the configurations in several ways.  We work over the complex field, let us summarize our results as follows. \par
Along the paper we construct in $\mathbb P^5$ an irreducible family of Morin configurations $F$ of length $k$ between $11$ and $20$. This family depends on $9$ moduli and defines a divisor in the moduli
space of finite Morin configurations. A general configuration has instead length $10$. For reasons soon to be evident, the members
of our family will be called \it Morin-Del Pezzo \rm configurations. Relying on the geometry of singular genus $6$ canonical curves, we describe these configurations of length $k \in [11,20]$. We prove that any smooth configuration $F$ of length $k \geq 16$ is Morin-Del Pezzo and moreover that:
\begin{thm}  \label{uni} Up to $\Aut \mathbb P^5$ a unique Morin configuration of planes in $\mathbb P^5$ exists having maximal cardinality $20$. \end{thm}
See sections from 6 to 9. The central core of the paper is dedicated to show several relations connecting Morin configurations of planes to the beautiful geometry of some classical projective varieties. Our methods relies indeed on these relations, which seem to be of independent interest. This includes: \medskip \par \noindent
(1) \it The geometry related to a quintic Del Pezzo surface and the Segre primal. \medskip \par \noindent \rm
(2) \it The family of threefolds $V \in \vert \mathcal O_{\mathbb P^2 \times \mathbb P^2}(2,2) \vert $ with isolated singularities.  \medskip  \par \noindent \rm
(3) \it Highly singular canonical curves of genus $6$ and possibly higher.
\rm \medskip \par
(1) To reasonably summarize these relations and our further work, let us consider a smooth quintic Del Pezzo surface $Y \subset \mathbb P^5$. The linear system $\mathbb P^4_Y$, of the quadrics through $Y$, is a $4$-space. By the way we prove the following result, see \ref{Segre}.
\begin{thm} The discriminant hypersurface  in $\mathbb P^4_Y$ is twice the Segre cubic $\Delta_Y$. \end{thm}
  Then we consider the union of the ten lines of $Y$. This is a stable canonical curve 
\begin{equation} 
C_{\ell} \subset Y \subset \mathbb P^5
\end{equation}
of genus $6$. The linear system $\mathbb P^5_{\ell}$ of the quadrics through $C_{\ell}$ is a $5$-space and $\mathbb P^4_Y$ is a hyperplane in it. It is known that the locus in $\mathbb P^4_Y$ of all quadrics of rank $\leq 4$ is union of five planes $P_1 \dots P_5$ of $\Delta_Y$.   Moreover, $\Sing C_{\ell}$ is a set of $15$ nodes and, for each $z \in \Sing C_{\ell}$, the linear system $P_z := \lbrace Q \in \mathbb P^5_{\ell} \ \vert \ z \in \Sing Q \rbrace$  is a plane. $P_z$ is not in $\mathbb P^4_Y$. Let
 $z_1, z_2 \in \Sing C_{\ell}$, one can show that $P_{z_1} \cap P_{z_2} \neq \emptyset$. Then it is possible to deduce as in section 8 that the mentioned planes define a Morin configuration
 \begin{equation} F_{\ell} :=  \lbrace P_1 \dots P_5, \ P_z \  z \in \Sing C_{\ell} \rbrace. \end{equation}
 In particular,  theorem \ref{uni}  can be also stated as follows.
  \begin{thm} $F_{\ell}$ is the unique Morin configuration of cardinality $20$ up to $\Aut \mathbb P^5$.   \end{thm}
(2) The Lagrangian Grassmannian $LG(10,20)$  and Morin configurations are strictly related. To be more precise let us fix some conventions, to be used throughout all the paper.
We assume $\mathbb P^5 = \mathbb P(W)$ and denote the natural symplectic pairing of $\wedge^3W$ as
\begin{equation}
\label{pairing}  w: \wedge^3 W \times \wedge^3 W  \to \wedge^6W. 
\end{equation}
Let $\mathbb G \subset \mathbb P(\wedge^3 W)$ be the Grassmannian of planes of $\mathbb P^5$. As is well known a closed scheme $F \subset \mathbb G$  is a Morin configuration iff its linear span is $\mathbb P(A)$, where $A$ belongs to the Lagrangian Grassmannian  $LG(10, \wedge^3 W)$ and \begin{equation} F = \mathbb P(A) \cdot \mathbb G. \end{equation}  We fix a point $u \in \mathbb G$: \it
for any finite $F$ considered $u$ will be a smooth point of $F$. \rm Let $P^{\perp}_u$ be the net of hyperplanes through the plane $P_u$, we also fix the notation:
\begin{equation}
\mathbb P^2 \times \mathbb P^2 := P_u \times P^{\perp}_u.
\end{equation}
This paper also relies on the construction, given in section $2$, where we associate to a Morin configuration $F$ a hypersurface of bidegree $(2,2)$ in $\mathbb P^2 \times \mathbb P^2$. Indeed, $F$ spans 
$\mathbb P(A)$ as above and $F$ is pointed by $u$. We show that the pair $(A,u)$ uniquely defines a hypersurface $V_A \subset \mathbb P^2 \times \mathbb P^2$ and prove the following theorem.
 \begin{thm} \label{biregular f_u} There exists a natural biregular map between $F - \lbrace u \rbrace$ and $\Sing V_A$. \end{thm}
This relates the study of Morin configurations $F$ of higher length to the study of hypersurfaces $V$ of bidegree $(2,2)$ in $\mathbb P^2 \times \mathbb P^2$
such that $\Sing V$ is finite. In particular let $V_{\ell} \subset \mathbb P^2 \times \mathbb P^2$ be the threefold associated to the maximal configuration $F_{\ell}$. In the paper we describe its very interesting geometry as follows. \medskip \par
$\circ$ \it $V_{\ell}$ contains a configuration of eight planes $a_i \times \mathbb P^2$ and $\mathbb P^2 \times b_j$, $1 \leq i,j \leq 4$, such that the sets $\alpha = \lbrace a_1 \dots a_4 \rbrace$ and $\beta = \lbrace b_1 \dots b_4 \rbrace$ are in general position in $\mathbb P^2$.\medskip  \par
$\circ$ Let $p: V_{\ell} \to \mathbb P^2$ be the first, (second),  projection and $\Gamma_p \subset \mathbb P^2$ its branch sextic. Then $\Gamma_p$ is the union of the singular conics of the 
pencil whose base locus is $\alpha$, ($\beta$). \medskip \par \rm
Let $\mathcal I$ be the ideal sheaf of the set $\lbrace (a_1,b_1) \dots (a_4,b_4) \rbrace \subset \mathbb P^2 \times \mathbb P^2$. Then $\vert \mathcal I(1,1) \vert$ defines a degree $2$ rational map $\pi: \mathbb P^2 \times \mathbb P^2 \to \mathbb P^4$, recently considered in \cite{FV, CKS}. Its branch divisor is the \it Igusa quartic threefold\rm, that is the dual of the Segre cubic. As a consequence of the mentioned results and of our description, it follows:
 \begin{thm} $V_{\ell}$ is the ramification divisor of $\pi$ and $\pi(V_{\ell})$ is the Igusa quartic. \end{thm}
 (3) Going back to the quintic Del Pezzo $Y$, let  $C \in \vert C_{\ell} \vert$ be a reduced singular curve and $\mathbb P^5_C$ the $5$-space of the quadrics through $C$.   As in the case of $C_{\ell}$ we can reconstruct from $\Sing C$, in the Grassmannian of planes of $\mathbb P^5_C$, the family of planes
  \begin{equation}
 F_C := \lbrace P_1 \dots P_5, P_z \ z \in \Sing C \rbrace.
 \end{equation}
where $P_1 \dots P_5$ are the nets of rank $4$ quadrics through $Y$ and $P_z \subset \mathbb P^5_C$ is the net of quadrics  which are singular at $z$. The next theorem is proven in section 8.   
  \begin{thm} Let $\Sing C$ be not in a hyperplane then $F_C$ is a Morin configuration.
\end{thm}
The special feature of $F_C$ is that $\lbrace P_1, \dots, P_5 \rbrace$ is a smooth linear section of the Grassmannian $\mathbb G_Y$
of planes of $\mathbb P^4_Y$, see \cite[8.5.3]{Do}. Then the corresponding points $p_1 \dots p_5$ only span a $3$-space.  Since a finite Morin configuration spans a $9$-space, it follows that $F_C$ has length  $k \geq 11$ and, moreover, $\Sing C$ necessarily spans $\mathbb P^5_C$. \par
\begin{defi} \it A  Morin-Del Pezzo configuration is a finite Morin configuration which contains with multiplicity one a $5$-tuple projectively equivalent to $\lbrace P_1 \dots P_5 \rbrace$. \end{defi} In sections 6, 7, 8 we 
construct an integral family whose members are the Morin-Del Pezzo  configurations and describe their properties.  Let $F = \mathbb P(A) \cdot \mathbb G$ be one of these and $V_A \subset \mathbb P^2 \times \mathbb P^2$ the bidegree $(2,2)$ hypersurface defined by $(A,u)$. We prove that: \it $F = F_C$ for some $C \in \vert C_{\ell} \vert$  and that $V_A$ contains a plane. \rm Then $V_A$ is rational and is reconstructed  from $C$ as follows. In the ambient space of $C$ the base locus of the net $P_5$ is a Segre product $\mathbb P^1 \times \mathbb P^2$.  Let $\mathcal J$ be the ideal sheaf of $C$ in it, then $\vert \mathcal J(2,2) \vert$ defines a rational map $q: \mathbb P^1 \times \mathbb P^2 \to \mathbb P^2$. Let $p: \mathbb P^1 \times \mathbb P^2 \to \mathbb P^2$ be the projection map, then:
\begin{thm}  $q \times p: \mathbb P^1 \times \mathbb P^2 \to \mathbb P^2 \times \mathbb P^2$ is a birational embedding with image $V_A$. \end{thm} 
These results are used to deduce theorem \ref{uni} and enumerate configurations. This is quickly done in section 9. Then some concluding remarks follow: we note that a stable canonical $C$ of genus $g \geq 7$ defines an analogous scheme of incident $(g-4)$-spaces in the dual of the space $\mathbb I_C$ of quadrics through $C$. That is $F'_C := \lbrace P_z, \ z \in \Sing C \rbrace$, where $P_z$ is the orthogonal of $\mathbb I_z := \lbrace Q \in \mathbb I_C \ \vert \ z \in \Sing Q \rbrace$. The involved dimensions satisfy the mentioned Zak's equality. This makes interesting the question: \begin{equation} \text{ \it when $F'_C$ is a Morin configuration and has maximal cardinality?}
\end{equation}
Here canonical graph curves, like $C_{\ell}$, could come into play. These are union of $2g-2$ lines and have $3g-3$ nodes. Each is uniquely defined by its dual associated graph. For $C_{\ell}$ this is the well known Petersen graph. 
 $$
\includegraphics[width=4cm]{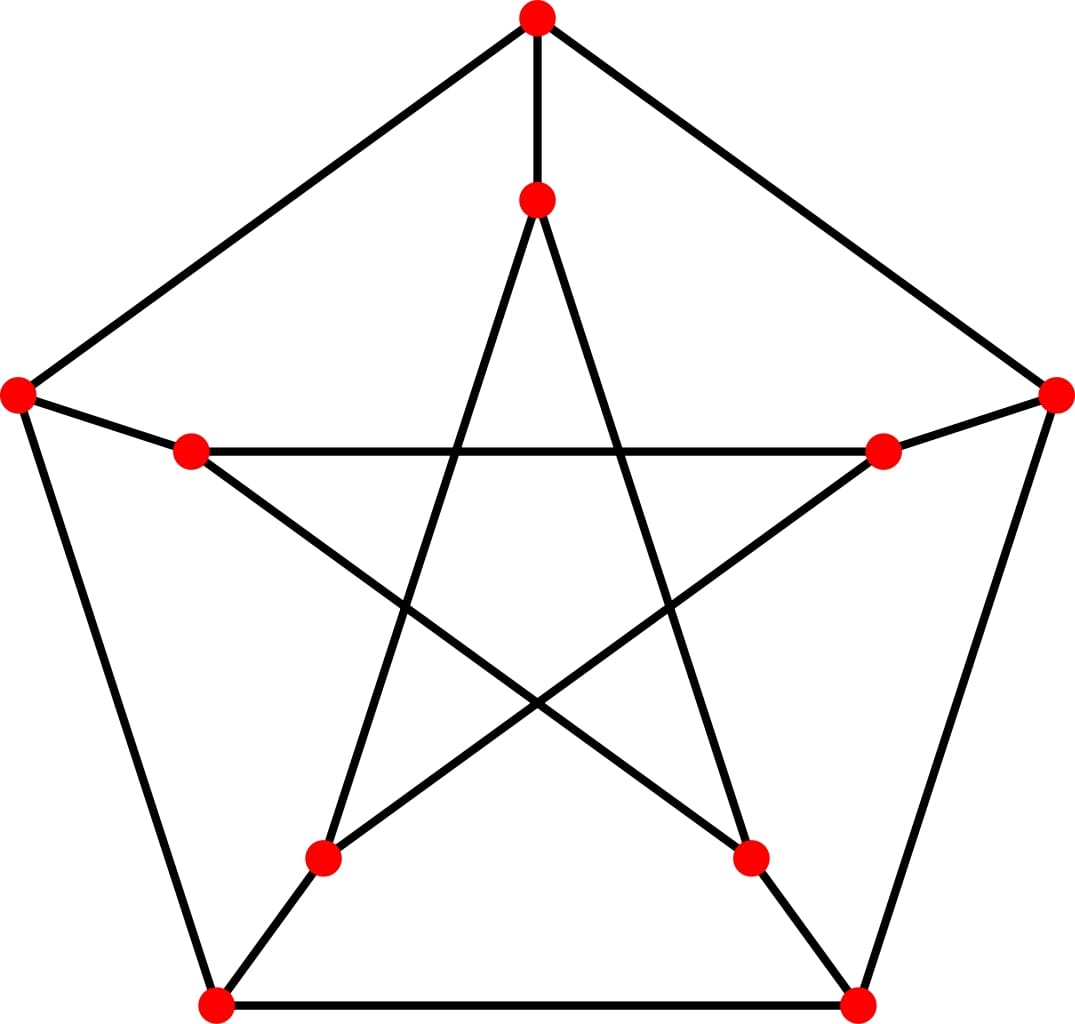}
$$
We discuss some example generalizing $C_{\ell}$ and chances that $3g-3$ be the maximal cardinality. In this paper we also revisit O'Grady's bound for $g = 6$ and  discuss realizations of singular plane sextics as $3 \times 3$ determinant of quadratic forms, see section 3, remark \ref{determ} and  \cite{O}. \medskip \par
\underline {Aknowledgements} The authors profited of useful comments from C.~Ciliberto, I. ~Dolgachev, A.~Iliev, M.~Kapustka, K.~Ranestad, C.~Shramov, F.~Viviani. \par
\underline  {Further notations} $\langle X \rangle$  linear span of $X$. [X] vector space generated by $X$.

   \section{Morin configurations of planes in $\mathbb P^5$ and $V$-threefolds} 
In this section we start studying finite Morin configurations of planes in $\mathbb P^5 = \mathbb P(W)$. We keep our conventions and begin from the point $u$, such that $[u]=U$, we have previously fixed in the Grassmannian $\mathbb G$. Notice that $u$ defines a natural filtration of $\wedge^3 W$, say
\begin{equation}
\label{filt} \wedge^3 W := W^0_u  \ \supset W^1_u \supset \ W^2_u \supset W^3_u := \wedge^3 U.
\end{equation} By definition $W^i_u \subset \wedge^3 W,  \ 1 \leq i \leq 3$,  is the  image of the pairing
\begin{equation} \wedge^iU \times \wedge^{3-i} W \to \wedge^3 W, \end{equation}  defined by the wedge product. Notice that $W^2_u / W^3_u$ is naturally isomorphic to the Zariski tangent space to $\mathbb G$ at $u$. Hence its projective completion is embedded as
\begin{equation}
\label{tang} \mathbb T_u := \mathbb P(W^2_u) \subset \mathbb P(\wedge^3W).
\end{equation}
By definition the   \it tangential projection of $\mathbb G$ from $u$ \rm is the linear map $ \mathbb P(\wedge^3 W) \to \mathbb P^9 $ of center the $9$-space $\mathbb T_u$. We will be more interested to its restriction
\begin{equation}
\tau: \mathbb H_u \to \mathbb P^8
\end{equation}
to the hyperplane $ \mathbb H_u := \mathbb P(W^1_u) \subset \mathbb P(\wedge^3 W)$.
We  point out that the target space of $\tau$ is $$ \mathbb P^8 = \mathbb P(W^1_u / W^2_u). $$
Moreover $\mathbb H_u$ cuts on $\mathbb G$ the codimension $1$ Schubert cycle $\mathbb G_u$ defined by $P_u$, that is
\begin{equation} 
\label{G_u}    \mathbb G_u := \lbrace e \in \mathbb G \vert P_e \cap P_u \neq \emptyset \rbrace = \mathbb H_u \cdot \mathbb G.
\end{equation}
It will be useful to describe $\tau \vert \mathbb G_u$. Let $e \in \mathbb G_u$ and $e \notin \mathbb T_u$, then  
$e = [u_1 \wedge e_2 \wedge e_3]$ so that $U \cap E =  [u_1]$ and  $\lbrace x \rbrace = P_u \cap P_e$ with $x = [u_1]$. This defines a rational map
\begin{equation}
\gamma: \mathbb G_u \to P_u
\end{equation}
such that $\gamma(e) := x$. Now
let $P^{\perp}_u \subset \vert \mathcal O_{\mathbb P^5}(1) \vert$ be the net of hyperplanes through $P_e$. It is also clear that $e$ uniquely defines an element $y \in P^{\perp}_u$. This is the hyperplane
in $\mathbb P^5$ generated by $P_u$ and the points $[e_2]$ and $[e_3]$. Notice also that we have $$ P^{\perp}_u = \mathbb P(\wedge^2 W/U) $$ via the assignment $y \to [(e_2 \mod U) \wedge (e_3 \mod U)]$. This defines a rational map
\begin{equation}
\gamma^{\perp}: \mathbb G_u \to P^{\perp}_u
\end{equation}
such that $\gamma^{\perp}(e) := y$. Finally, since $\mathbb P^8 = \mathbb P(W^1_u / W^2_u)$, we have also a natural map
\begin{equation}
s: P_u \times P^{\perp}_u \to \mathbb P^8
\end{equation}
such that $s(x,y) := [u_1 \wedge e_2 \wedge e_3 \mod W^2_u]$. Leaving some details to the reader, we conclude that $s$ is the Segre embedding of $P_u \times P^{\perp}_u$. Hence the next lemma follows.
  \begin{prop} $\tau: \mathbb G_u \to \mathbb P^8$ factors as in the next diagram:
$$  \begin{CD}
{\mathbb G_u} @>{\gamma \times \gamma^{\perp}}>> {P_u \times P^{\perp}_u} @>s>> {\mathbb P^8} \\
\end{CD}
$$
\end{prop}
Let $(x,y)$ be as above then $x = [u_1]$. Moreover  $y$ defines in $W$ the codimension $1$ vector space $W_y = [U,  e_2,  e_3]$. Then the fibre of $\tau \vert \mathbb G_u$ at $(x,y)$ is the family
of planes 
\begin{equation}
\lbrace P \ \vert \ x \in P \subset \mathbb P (W_y) \rbrace.
\end{equation}
With some more effort, one can show that such a fibre is naturally embedded as the Pl\"ucker quadric of the Grassmannian of lines of $\mathbb P(W_y / [u_1])$.
\medskip \par
Now we start dealing with maximal isotropic spaces $A$ of $w$.  
\begin{prop} Let $A$ be such a space then: $ u \in \mathbb P(A) \Leftrightarrow A \subset W^1_u. $  \end{prop}
\begin{proof} Assume $A \subset W^1_u$. Since $W^1_u$ is orthogonal to $\wedge^3U$, the space $A + \wedge^3U$ is isotropic. Since $A$ is maximal isotropic then $\wedge^3 U \subset A$ 
and $u \in \mathbb P(A)$. The converse is obvious. \end{proof}
Next \it we fix the following assumptions \rm on the maximal isotropic space $A$.
\begin{itemize} \it \label{conv}
\item[$\circ$]  $ A \subset W^1_u$, \medskip \par
\item[$\circ$]  $ A \cap W^2_u = \wedge^3 U$.
\end{itemize}
Equivalently $u \in \mathbb P(A)$ and the intersection scheme $\mathbb P(A) \cdot \mathbb G$ is smooth and $0$-dimensional at $u$. That is a cheap restriction with respect to our goals. We will be mainly interested in
the following loci in $LG(10,20)$, to be repeatedly considered.
\begin{defi}
\begin{itemize} \ \par
\item[$\circ$] $\mathcal A := \lbrace A \in LG(10,20) \ \vert \ \text {\it the scheme} \ \mathbb P(A) \cdot \mathbb G \ \text {\it is finite} \rbrace$. \medskip \par
\item[$\circ$] $\mathcal A^c := \lbrace A \in \mathcal A \ \vert \ \text{\it $\mathbb P(A) \cdot \mathbb G$ is a Morin configuration} \rbrace$.
\end{itemize}
\end{defi}
Under our assumptions $\mathbb P(A)$ contains $u$, now we consider the restriction
\begin{equation}
\label{tau_A} \tau_A: \mathbb P(A) \to \mathbb P^8
\end{equation}
of $\tau$ to $\mathbb P(A)$. Since we have $\mathbb P(A) \cap \mathbb T_o = \lbrace u \rbrace$, it is clear that $\tau_A$ is just the projection of $\mathbb P(A)$ from its point $u$ and that the image of $\tau_A$ is
 $\mathbb P^8$. Since $A$ is isotropic, a quadratic section of $P_u \times P^{\perp}_u$ is intrinsically associated to $A$ as follows. \par Let $y  \in P^{\perp}_u$ then $y = [(e_2 \mod U) \wedge (e_3 \mod U)]$ for some vectors $e_2, e_3 \in W-U$. It is easy to describe the $3$-space $\tau_A^{-1}(P_u \times \lbrace y \rbrace)$. Indeed, let $a \in A$ then, as for any vector of $W^1_u$, we can write $a = a_1 + a_2 + a_3$, where $a_i$ is in the image of the previously considered  pairing $\wedge^i U \times \wedge^{3-i} W \to \wedge^3 W$, $i = 1,2,3$. It is therefore clear that
 \begin{equation}
 [a] \in \tau_A^{-1}(P_u \times \lbrace y \rbrace) \ \Longleftrightarrow a_1 \in U \wedge (\wedge^2 [e_2, e_3]).
 \end{equation}
 Let $[a], [a'] \in \tau_A^{-1}(P_u \times \lbrace y \rbrace)$ so that $a = a_1 + a_2 + a_3$ and $a' = a'_1+a'_2+a'_3$, then 
 \begin{equation}  a \wedge a' = a_1 \wedge a'_1 + a_1 \wedge a'_2 + a_2 \wedge a'_1 = 0. \end{equation}  
 Let $a_1 = v_1 \wedge v_2 \wedge v_3$ and $a'_1 = v'_1 \wedge v'_2 \wedge v'_3$, we can assume $v_1, v'_1 \in U$. Moreover we have $[U, v_2, v_3] = [U, v'_2, v'_3] = [U, e_2, e_3]$. Since this vector space has dimension $5$, its vectors $v_1, v_2, v_3, v'_1, v'_2, v_3$ are linearly dependent. This implies $a_1 \wedge a'_1 = 0$ so that
\begin{equation} \label{eq} a_1 \wedge a'_2 = a'_1 \wedge a_2. \end{equation}  Let $A_y \subset A$ be the subspace such that $\mathbb P(A_y) = \tau_A^{-1}(P_u \times \lbrace y \rbrace)$ and let $a, a' \in A_y$. Then
 the above equality \ref{eq} defines a symmetric bilinear form 
\begin{equation}
< , >_y: A_y \times A_y \to \mathbb C
\end{equation}
such that $<a,a'>_y := a_1 \wedge a'_2$. Let $A_x \subset A$ be such that $\mathbb P(A_x) = \tau_A^{-1}(\lbrace x \rbrace  \times P^{\perp}_u)$. In the same way, putting $<a,a'>_x := a_1 \wedge a'_2$,  we obtain  a symmetric bilinear map
\begin{equation}
< , >_x: A_x \times A_x \to \mathbb C.
\end{equation}
We omit some details. Since $<a,a'>_x = <a,a'>_y$, the construction defines a vector $v_A \in H^0(\mathcal O_{P_u \times P^{\perp}_u}(2,2))$ whose restrictions to
$P_u \times \lbrace y \rbrace$ and $\lbrace x \rbrace \times P^{\perp}_u$ respectively are the quadratic forms $< , >_y$ and $< , >_x$. \par Following some use we say that a bidegree $(2,2)$ hypersurface in $\mathbb P^2 \times \mathbb P^2$ is a Verra threefold, for short a \it $V$-threefold. \rm As we will see, $v_A$ is not zero so that $\di (v_A)$ is a $V$-threefold of $P_u \times P^{\perp}_u$.  Let us introduce the following definitions.
\begin{defi} \label{v_A} \it $V_A := \di(v_A)$ is the $V$-threefold associated to $A$. \end{defi}
\begin{defi} \label{F_A} \it $F_A := \mathbb P(A) \cdot \mathbb G$ is the scheme of incident planes of $A$. \end{defi}
\em From now on we will assume, up to different advice, that $F_A$ is finite. \rm \medskip \par
Now we describe $F_A$ in terms of the singular locus of $V_A$. Let $F'_A := F_A - \lbrace u \rbrace$ and let $e \in \mathbb P(A) - \lbrace u \rbrace$. We consider representations of $e$ 
as $e = [a_1 + a_2 + a_3]$, with $a_1, a_2, a_3$ as above. It is clear that the following condition are equivalent: \medskip \par
\begin{enumerate} \it
\item the line joining $u$ to $e$  intersects $F'_A$,
\item a representation  of $e$ satisfies $a_2 = 0$.
\end{enumerate} \medskip \par
Indeed (1) is equivalent to $[\wedge^3 U, a_1+a_2+a_3] = [\wedge^3 U, b_1]$, for some decomposable $b_1 \in W^1_u - W^2_u$, that is $e = [b_1 + b_3]$ for some $b_3 \in \wedge^3 U$. Notice also that:
\begin{lemm} The map $\tau_{ \vert F'_A}: F'_A \to P_u \times P^{\perp}_u \subset \mathbb P^8$ is biregular to its image. \end{lemm}
\begin{proof} We have $\tau(F'_A) \subset \tau(\mathbb G_u) \subset P_u \times P^{\perp}_u$.  To prove that $\tau_{\vert F'_A}$ is biregular to $\tau(F'_A)$ consider any scheme $\zeta \subset F'_A$
of length $2$. We have $\zeta \subset \mathbb P(A)$. Moreover the restriction of $\tau$ to $\mathbb P(A)$ is the projection from $u$. Hence $\tau_{\vert \zeta}$ is not biregular to its image iff the line
$\langle \zeta \rangle$ contains $u$. Since $u \notin \zeta$, this is equivalent to say that the scheme $\langle \zeta \rangle \cdot \mathbb G$ has length $\geq 3$. Then
it follows $\langle \zeta \rangle \subset \mathbb G$, because $\mathbb G$ is intersection of quadrics, and hence $F_A$ is not finite: against our assumption. This implies the statement. \end{proof}
Now let us consider the cone $C(F_A) := \tau^*_{\vert \mathbb P(A)}\tau(F_A)$. Then $C(F_A)$ is a cone of vertex $o$ over $F'_A$ and it is defined by the equation $a_2 = 0$, that is
\begin{equation}
C(F_A) = \lbrace [a_1 + a_2 + a_3] \in \mathbb P(A) \ \vert \ a_2 = 0 \rbrace.
\end{equation}
Of course the condition defines as well the embedding $\tau(F'_A) \subset P_u \times P^{\perp}_u$. Now we study $\tau(F'_A)$. To this purpose let $e = [a_1 + a_2 + a_3] \in \mathbb P(A)$ as usual. If $e \in C(F_A)$ then we have $\tau(e) = (x,y) \in P_u \times P^{\perp}_u$. At first we remark that the condition $a_2 = 0$ is precisely equivalent to the property that the polar forms
\begin{equation}
<\cdot, a_2 >_y: A_x \to \mathbb C \ \text{and} \ <\cdot , a_2 >_x: A_y \to \mathbb C
\end{equation}
of the vector $a_1+a_2+a_3$ are identically zero. This immediately translates in the following simple condition on a point $o := (x,y) \in P_u \times P^{\perp}_u$: \begin{equation} \label{char}  \text{ \it  both the planes
$P_u \times \lbrace y \rbrace$ and $\lbrace x  \rbrace\times P^{\perp}_u$ are tangent to $V_A$ at $o$. } \end{equation}  
Since these planes generate the embedded tangent space in the Segre embedding of $P_u \times P^{\perp}_u$, it follows that $o \in \tau(F'_A)$ iff $o$ is singular for $V_A$. In order to have more
precision let us write explicitly the equations of $\tau(F'_A)$. Under our notation we have
$$
P_u \times P^{\perp}_u = \mathbb P^2 \times \mathbb P^2 \subset \mathbb P^8.
$$
On $\mathbb P^2 \times \mathbb P^2$ we fix projective coordinates $(x_1:x_2:x_3) \times (y_1:y_2:y_3)$  defining the point $(x,y)$ and then we consider the equation  $f$ of $V$. Therefore we have  
\begin{equation}
f = \sum a_{ij}x_ix_j = 0,
\end{equation}
where the $a_{ij}$'s are quadratic forms in $y$.  By the condition \ref{char} the partials
$$
f_{x,i} := \frac {\partial f}{\partial x_i } \ , \  f_{y,i}Ê:= \frac {\partial f}{\partial y_i } \ , \ i = 1,2,3,
$$
define $\tau(F'_A)$, so the next theorem follows.
\begin{thm}\label{main}  $F'_A$ is biregular to
the scheme defined by the above derivatives i.e.~to the singular locus of $V_A$. \end{thm}
We are grateful to M. Kapustka for discussions around this result. We remark that $\tau(F'_A)$ fits in the standard exact sequence
\begin{equation}
0 \to T_{V_A} \to T_{\mathbb P^2 \times \mathbb P^2 \vert V_A} \to \mathcal O_{V_A}(2,2) \to \mathcal O_{F'_A} \to 0
\end{equation}
of tangent and normal sheaves realizing the singular locus of $V_A$.  This complete the proof of
theorem \ref{biregular f_u}.
 \section{$V$-threefolds with isolated singularities}
Continuing  in the same vein we consider now a $V$-threefold $V \subset \mathbb P^2 \times \mathbb P^2$ such that $\Sing V$ is finite. We want to discuss more on $\Sing V$. Let us consider the projections
\begin{equation} \label{projections}
\begin{CD}
{\mathbb P^2} @<{\pi_x}<< { V} @>{\pi_y}>> {\mathbb P^2} \\
\end{CD}
\end{equation} 
and the schemes $R_x\subset V$ and $R_y\subset V$ respectively defined by the ideals
\begin{equation} \label{ramif}  \mathcal J_{V,x} := (f_{x,1}, f_{x,2}, f_{x,3}) \ , \ \mathcal J_{V,y} := (f_{y,1}, f_{y,2}, f_{y,3}). \end{equation}
It is clear that $R_x$ and $R_y$ are the ramification schemes respectively of $\pi_y$ and $\pi_x$. Their supports are the loci where the  tangent maps $d\pi_y$ and $d\pi_x$ have rank $\leq 1$. Now,
in the Chow ring $\CH^*(\mathbb P^2 \times \mathbb P^2)$, let $h_x$ and $h_y$ be respectively the classes of the pull-back of a line by $\pi_x$ and $\pi_y$. 
 Then it is very easy to see that $f_{x,i}$ and $f_{y,i}$ define divisors 
 $$
D_{x,i} := \di (f_{x,i}) \in \vert h_x + 2h_y \vert  \ , \ D_{y,i} := \di (f_{y,i}) \in \vert 2h_x + h_y \vert.
$$
The next properties we show for $R_x$ are true for $R_y$ with the same arguments.
\begin{lemm} \label{lemma3.3} Let $o \in \Sing V$. If  the plane $\mathbb P^2 \times \pi_y(o)$ is not in $V$ then $o \in \Sing R_x$.
\end{lemm}
\begin{proof} Cf. \cite{FV} 1.2. Let $(x_1: x_2: x_3) \times (t_1, t_2)$ be coordinates on $\mathbb P^2 \times \lbrace y_3 \neq 0 \rbrace $ so that $t_i := \frac{y_i}{y_3}$ for $i = 1,2$ and  $o$ is 
$(0:0:1) \times (0,0)$.  On it the equation of $V$ is
$$
qx_3^2 + (at_1x_1 + bt_2x_1+ ct_1x_2 + dt_2x_2)x_3 + p = 0,
$$
where $q \in \mathbf C[t_1,t_2]$, $p \in \mathbf C[x_1, x_2]$ are quadratic forms. The partials $\frac{\partial}{\partial x_i}$, $i = 1,2,3$, give local equations of $R_x$. In affine coordinates $u_1 := \frac{x_1}{x_3}$, $u_2 := \frac{x_2}{x_3}$, these equations are
$$
2q+at_1u_1+bt_2u_1+ct_1u_2+dt_2u_2=
$$
$$
at_1+bt_2+2p_{11}u_1+2p_{12}u_2= ct_1+dt_2+2p_{22}u_2+2p_{12}u_2 = 0,
$$
where $p = \sum p_{ij}u_iu_j$. Clearly the tangent space $T_{R_x,o}$ is defined by the latter two equations so that $\dim T_{R_x,o} \geq 2$. By theorem \ref{surf-component}  the only irreducible
surfaces in $R_x$ are planes $\mathbb P^2 \times y$. Since the only one through $o$ is not, it follows that $o \in \Sing R_x$. 
 \end{proof}
 \begin{thm} \label{no plane} Assume the intersection scheme $R_x= D_{x,1} \cdot D_{x, 2} \cdot D_{x,3}$ is proper, then the length of the singular locus $\Sing V$ is $\leq 15$.
 \end{thm}
 \begin{proof}  Since it is proper, $R_x$ is complete intersection of the divisors $D_{x,1}, D_{x,2}, D_{x,3}$  of class $h_x + 2h_y$. Hence one computes that $R_x$ has arithmetic genus $10$ and 
 class $6h_x^2h_y + 12h_xh_y^2$ in $\CH^*(\mathbb P^2 \times \mathbb P^2)$. Since $\Sing V$ is finite, no component of $R_x$ is a fixed component of the net of divisors generated by $D_{y,1}, D_{y,2}, D_{y,3}$. Hence  an 
 element $D$ of this net intersects $R_x$ properly and $\Sing V$ is embedded in the finite scheme  $R_x \cdot D$. By lemma \ref{lemma3.3} each point $o \in \Sing V$ is singular for  $R_x$. Then, 
 since $R_x \cdot D$ has length $30$ and its multiplicity is $\geq 2$ at each $o \in \Sing V$, the length of $\Sing  V$ is $\leq 15$. 
 \end{proof}
\begin{lemm} If $\Sing V$ is finite the discriminant of $\pi_y\vert V: V \to \mathbb P^2$ is a reduced curve. \end{lemm}
\begin{proof} Let $f = \pi_y \vert V$. From the finiteness of $\Sing V$ and generic smoothness it follows that the discriminant of $f$ is a curve. Assume $B$ is a non reduced, irreducible component
of it. Let $y \in B$ be a general point then $\Sing V \cap f^*(y) = \emptyset$. Moreover it follows from \cite{B3} that $f^*(y)$ is a conic of rank $1$. Let $S \subset V$ be the closure of the union of the
lines $f^{-1}(y)$, where $y \in B$ is general. Then $f: S \to B_{red}$ is a $\mathbb P^1$-bundle and $f^*B_{red}$ has multiplicity $2$ along $S$. In particular there exists an affine open set $U = \Spec R \subset \mathbb P^2$ so
that $U \cap B \neq \emptyset$ and the equation of $V$ in $U \times \mathbb P^2$ is $da^2 - bc$, where $a,b,c,d \in R[x_1,x_2,x_3]$. Moreover $b \in R$ is the equation of $B_{red}$ in $U$, 
$d \in R$ and $a, c \in R[x_1,x_2,x_3]$ are forms respectively of degree $1$ and $2$ in $(x_1,x_2,x_3)$. Since $d \notin (b)$ and $V$ is irreducible, we can assume $d = 1$ up to shrinking $U$.
Now consider in $U \times \mathbb P^2$  the set $Z = \lbrace a^2 = b = c = 0 \rbrace$. It is clear that $Z$ is non empty and hence of dimension $1$. Moreover, $Z$ is contained in $\Sing V$: this
contradicts the finiteness of $\Sing V$. \end{proof}
It easily follows that the only surfaces possibly contained in $R_x \cup R_y$ are planes.
   \begin{thm} \label{surf-component} Let $S \subset R_x$ be an irreducible surface then $S = \mathbb P^2 \times o$, for some $o \in \mathbb P^2$. \end{thm}
  \begin{proof}  $\pi_y(S)$ is an irreducible component of the discriminant curve $\Gamma$ of $\pi_y \vert V$. By the lemma $\Gamma$
  is reduced Assume $\pi_y(V)$ is a curve, then the previous lemma and its proof imply that the general fibre of $V$ over $\pi_y(S)$ is a conic of rank $2$. This is impossible because implies $\dim S = 1$.
 Hence $\pi_y(S)$ is a point and $S$ is a plane,  fibre of $\pi_y \vert V$.  \end{proof}
 Now we assume that $R_x$ contains a plane $P := \mathbb P^2 \times o$.  
  \begin{prop}  \label{no fixed} Let $b := P \cdot \Sing V$ then $b$ is the base locus of a pencil of conics.
 \end{prop}
\begin{proof} We can assume that $P = \lbrace y_1 = y_2 = 0 \rbrace$. Then the equation of $V$ is:
 $$ f = q_{11}y_1^2 + q_{22}y_2^2 + q_{12}y_1y_2 + q_{13}y_1y_3 + q_{23}y_2y_3. $$
Restricting the derivatives $f_{y,1}, f_{y,2}, f_{y,3}$ to $P_u$ we conclude that 
$$ \Sing V \cdot P = \lbrace y_1 = y_2 = q_{13} = q_{23} = 0 \rbrace.$$ Hence $b$ is the base locus of 
  the pencil of conics $\lambda q_{13} + \mu q_{3} = 0$.
  \end{proof}
\begin{rem}Ê\label{Csquare} The locus $b$ is the complete intersection $\lbrace q_{13} = q_{23}Ê= 0 \rbrace$.  Assume for simplicity that $b$ is smooth and let $\sigma: Y \to \mathbb P^2$ be the blowing up of $o = \pi_y(P)$. Then a  standard  resolution $\phi: \tilde V \to V$ of 
$\Sing V$ at $b$ is provided by the Cartesian square
\begin{equation}
\label{vertplane} \begin{CD}
 {\tilde V} @>{\phi}>> V \\
 @V{\tilde \pi_y}VV @VV{\pi_y}V \\
{Y} @>{\sigma}>> {\mathbb P^2.} \\
\end{CD}
\end{equation}
Let $\tilde P$ be the strict transform of $P$ and $E = \sigma^{-1}(o)$. Then $\tilde \pi_y(\tilde P) = E$ and the morphism $\tilde \pi_y: \tilde P \to E$ is defined by the pull-back of the pencil of conics of $\lambda q_{13} + \mu q_{23} = 0$.
 \end{rem}
  It is now useful to define the sets
$$
\mathcal P_x := \lbrace x \in \mathbb P^2 \ \vert \ x \times \mathbb P^2 \subset V \rbrace \ , \ \mathcal P_y := \lbrace y \in \mathbb P^2 \ \vert \ \mathbb P^2  \times y \subset V \rbrace. 
$$
\begin{defi} \it $t_x(V)$ and $t_y(V)$ are the cardinalities of $\mathcal P_x$, and $\mathcal P_y$. \end{defi}
\begin{lemm} \label{gen-pos-lemma} Both $\mathcal P_x$ and $\mathcal P_y$ have at most four points and no three are collinear.  
 \end{lemm}
 \begin{proof}  Assume $\mathcal P_x$ contains five points $o_1 \dots o_5$ and let $C$ a conic through these. It is easy to see that then $V$ properly contains  $\mathbb P^2 \times C$: against the irreducibility 
 of $V$.  The same proof applies to $V$ and $\mathbb P^2 \times L$, where $L$ is a line through three  points of $\mathcal P_x$. \end{proof}
 It follows from the previous results that $R_x$ and $R_y$ admit the decompositions $$ R_x = R_x^1 \cup R_x^2 \ \ \text {and} \ \ R_y =R^1_y + R^2_y$$ 
 where $R^1_x, R^2_y$ are curves and $R^2_x, R^2_y$ disjoint union of planes respectively of class $h_x^2, h_y^2$.  In what follows we assume, for each plane $P$ in $R^2_x$, that $b = P \cdot \Sing V$ is smooth. Assuming this we now give an alternative proof of O'Grady's bound $\vert \Sing V \vert \leq 20$, \cite{O}. This will be useful for further purposes. Dropping the smoothness assumption for $b$,
 one has to extend the argument of the proof to any $b$ which is complete intersection of two conics. Since the bound is known, we avoid to address the singular cases.  In particular the next theorem suggests that a $V$-threefold $V$ such that $\vert \Sing V \vert = 19$ has to contain four disjoint planes.

 \begin{thm} \label{excess}  Let $t := \mi \lbrace t_x(V), t_y(V) \rbrace$ then $ \vert \Sing V \vert \leq 15 + t $ and $t \leq 4$. \end{thm}
 \begin{proof} Let $t = 1$ then $R^2_x$ is a plane $\mathbb P^2 \times o$.  Consider a general member $D$ of the net generated by $D_{y,1}, D_{y,2},D_{y,3}$. As in the proof of theorem \ref{no plane},
 we can assume that $D$ intersects the curve $R^1_x$ properly and that $B := D \cdot R^2_x$  is a smooth conic containing $b := \Sing V \cdot R^2_x$. We can also assume that $B$ is disjoint from $R^2_x \cdot R^1_x$. Indeed, as follows from remark \ref{Csquare}, this consists of the singular points of the singular conics through $b$. 
Now we consider the intersection scheme $D \cdot R_x$. This is defined by 3 divisors of class $h_x + 2h_y$ and one of class $2h_x + h_y$. In \ref{surf-component} this intersection was proper and hence of length $30$. Here it is not proper and $B$ is the excess intersection scheme. Since $b$ is smooth one can check that $B$ is a smooth irreducible component of $D \cdot R_x$.
Applying excess intersection formula to $B$, \cite[6.3]{F}, one computes that $D \cdot R_x = B \cup Z$, where $Z$ has length $24$ and $B \cap Z = \emptyset$. Now, arguing as in the proof of \ref{surf-component}, each 
$o \in Z$ has multiplicity $\geq 2$. This implies that the cardinality of $\Sing V$ is $\leq 12 + \deg b = 16$ and proves the statement for $t = 1$. The argument easily extends to $t \leq 4$.
\end{proof}
\begin{rem} Consider a general $\overline D$ of the net generated by $D_{x,1}, D_{x,2}, D_{x_3}$. Then $S = D \cdot \overline D$ is a complete intersection of
class $2h_x^2 +5h_xh_y + 2h_y^2$. Assume $S$ is integral with at most isolated nodes, which is the general case. Then $S$ is a K3 surface through $R^1_x \cup B$. Let $\sigma: S' \to S$ be
its minimal desingularization and $B'$ and $H'$ the pull-back of $B$ and $H \in \vert \mathcal O_S(1,2)\vert$. It is easily seen that $Z$ has length
$(H'-B')^2 = 30 + B^{'2} -2H'B' = 24$. This recovers the above excess intersection formula for $B$, cfr. \cite[13.3.6]{EH}.
\end{rem}
 \section{Highly singular V-threefolds and the Igusa quartic}
Before introducing the main family of finite Morin configurations to be considered, and explicitly reconstruct in it the unique one of maximal cardinality, we already use the previous results to describe its associated $V$-threefold and its relation to a well known threefold in $\mathbb P^4$, namely the Igusa quartic. 
\begin{defi} \it Let $o \in \Sing V$, we say that $o$ is a tangential singularity if
$$
(\pi_x(o) \times \mathbb P^2) \cup (\mathbb P^2 \times \pi_y(o)) \subset V.
$$
Moreover we denote by $t(V)$ the number of these singularities on $V$.
\end{defi}
Assume $V$ is defined by $A$, so that $F_A = \mathbb P(A) \cdot \mathbb G$. Then, by the latter theorem,
$$ \vert F_A \vert \leq 16 + t(V) $$
and $0 \leq t(V) \leq 4$. This gives a constructive way to produce families of finite Morin configurations of higher length $\ell$ in the range
$16 \leq \ell \leq 20$. Indeed, let $t := t(V) \geq 1$ then $\Sing V$ is necessarily endowed with a set of tangential singularities
$$ O_t := \lbrace o_1 \dots o_t \rbrace \subset \Sing V $$
such that the projection maps $\pi_x: O_t \to \mathbb P^2$ and $\pi_y: O_t \to \mathbb P^2$ are injective. In particular $V$ contains $2t$ distinct planes, say $ \lbrace u_i \rbrace \times \mathbb P^2 \ , \ \mathbb 
P^2 \times \lbrace v_j  \rbrace $ with $1 \leq i,j \leq t$. Then a set $O_t$ as above is $ \lbrace o_i := (u_i,v_i), \ i = 1 \dots t \rbrace$. By lemma \ref{gen-pos-lemma} the sets
$$
\lbrace u_1 \dots u_t \rbrace \ , \  \lbrace v_1 \dots v_t \rbrace
$$
are sets of distinct points so that no three are collinear.  To construct configurations of length $\ell \geq 16$ we consider the union of planes
$$
U_t := \bigcup_{1 \leq i,j \leq t} (\lbrace u_i \rbrace \times \mathbb P^2) \cup (\mathbb P^2 \times \lbrace v_j  \rbrace)
$$
and its ideal sheaf $\mathcal I_t$  in $\mathbb P^2 \times \mathbb P^2$. This defines the linear system of $V$-threefolds
$$
\vert \mathcal I_t(2,2) \vert.
$$
The case $t = 4$ leads to Morin configurations of length $\ell \geq 16$, in particular to the maximal one with $20$ planes. In what follows we assume $t = 4$. Since
the points $u_1 \dots u_4$ and $v_1 \dots v_4$ are in general position, we can fix coordinates $(x,y)$ on $\mathbb P^2 \times \mathbb P^2$ so that $(u_i,v_i)$ is in the diagonal $\lbrace x - y = 0 \rbrace$. We can also assume that \begin{equation} u_1 = (1:0:0) ,  u_2 = (0:1:0) ,  u_3 = (0:0:1)  , u_4 = (1:1:1). \end{equation} Let $q_1(x)$ and $q_2(x)$ be quadratic forms in $x$ generating the ideal of $\lbrace u_1 \dots u_4 \rbrace$. Then $q_1(y)$ and $q_2(y)$ generate the ideal of $\lbrace v_1 \dots v_4 \rbrace$ and the next theorem easily follows.
 \begin{thm} $ \vert \mathcal I_4(2,2) \vert$ is the $3$-dimensional linear system  
$$
\lambda q_1(x)q_1(y) + \mu q_2(x)q_2(y) + \nu q_1(x)q_2(y) + \rho q_2(x)q_1(y) = 0.
$$
Let $V \in \vert \mathcal I_4(2,2) \vert$ be general then $\Sing V$ is the set of $16$ tangential singulartities $$ \lbrace o_{ij} =(u_i,v_j), \ 1 \leq i,j  \leq 4 \rbrace.$$
\end{thm}
Later in this paper we will see that the branch sextic $\Gamma$ of $\pi_y: V \to \mathbb P^2$ is the union of three conics of the pencil $\lambda q_1(y) + \mu q_2(y) = 0$. The most interesting case of $V$ arises when $\Gamma$ is the union of the three singular conics of the pencil, that is  
$$
\Gamma = \lbrace y_1y_2y_3(y_1-y_2) (y_1-y_3)(y_2-y_3) = 0 \rbrace.
$$
Then it turns out that $V$ has 19 ordinary double points: the $16$ tangential singularities and $3$ other points, one over each double point of $\Gamma $.   We will also show that a unique $V$ satisfies $\vert \Sing V \vert = 19$ and that it is defined by a complete Morin configuration of $20$ planes in $\mathbb P^5$, which is unique as well. For reasons to be made clear in the end of this section, we fix for such a $V$ the notation $V^{ram}$. Its equation is  
\begin{equation}
\label{V_{ram}}
 (x_1-x_2)x_3(y_1-y_2)y_3 + x_1(x_2-x_3)y_1(y_2-y_3) + x_2(x_1-x_3)y_2(y_1-y_3) = 0.
\end{equation}
We continue this section by some constructions useful to put  $V^{ram}$ in its due geometric perspective.   Let $O_4 \subset \mathbb P^2 \times \mathbb P^2$ be the set of four points as above and let $\mathcal B_4$ be the linear system of $V$-threefolds singular at $O_4$.   We consider the linear projection of 
 center $O_4$
\begin{equation} \label{phi}
\phi: \mathbb P^2 \times \mathbb P^2 \to \mathbb P^4.
\end{equation}
of the Segre embedding $\mathbb P^2 \times \mathbb P^2 \subset \mathbb P^8$. The map $\phi$ is defined by the linear system $\vert \mathcal I(1,1) \vert$, where $\mathcal I$ is the ideal sheaf of $O_4$ in $\mathbb 
P^2 \times \mathbb P^2$. The base scheme of  $\vert \mathcal I(1,1) \vert$ is precisely $O_4$. Then, since the Segre product $\mathbb P^2 \times \mathbb P^2$ has degree six, it follows $\deg \phi = 2$. 
\par The ramification divisor of $\phi$ is strictly related to the subject of this paper and to a very well known  threefold and its dual. We recall that \it the Segre primal \rm
 is the unique, up to projective equivalence,  cubic threefold $\Delta$ whose singular locus consists of ten double points, which is the maximum for a cubic hypersurface with isolated singularities  in
a $4$-space. Of equivalent interest is its dual hypersurface
\begin{equation}
\label{Igusa} \Delta^* \subset \mathbb P^4.
\end{equation}
This is in turn a quartic threefold which is very well known. It is \it the Igusa quartic, \rm see e.g. \cite{CKS} and \cite{Do}. In particular, in the recent paper \cite{CKS}, it is shown that:
\begin{thm}  $\Delta^*$ is the branch divisor of $\phi$. \end{thm} 
Now let us consider as in \ref{V_{ram}} the threefold $V^{ram}$, which defines the unique complete Morin configuration of 20 planes in $\mathbb P^5$. Relying on its equation, and on the equations of $\phi$,  it is not difficult to compute the image of $V^{ram}$ by $\phi$ and conclude as follows.
 \begin{thm} $V^{ram}$ is the ramification of $\phi$ and $\phi(V^{ram})$ is the Igusa quartic. \end{thm}
 
\section{Del Pezzo 5-tuples of planes and the Segre primal}
In what follows $\mathbb G_{\mathbb P^4}$ is the Grassmannian of planes of $\mathbb P^4$ embedded by its Pl\"ucker map, then $\deg \mathbb G_{\mathbb P^4} = 5$. Let
us consider \it any \rm transversal $0$-dimensional linear section \begin{equation}
\label {h}  h :=  \lbrace h_1 \dots h_5 \rbrace \subset \mathbb G_{\mathbb P^4},
 \end{equation}
 then $h$ spans a $3$-space. It is  known that its points are in general position in $ \langle h \rangle$.
\begin{defi} \it We say that $h$ is a Del Pezzo 5-tuple of planes of $\mathbb P^4$. \rm \end{defi}
All Del Pezzo $5$-tuples are projectively equivalent. So it is not restrictive fixing a Del Pezzo $5$-tuple of special geometric interest as follows.
Let $Y \subset \mathbb P^5$ be a smooth quintic Del Pezzo surface and $\mathcal I_Y$ its ideal sheaf. Then $\vert \mathcal I_Y(2) \vert$ is a $4$-space endowed with a natural Del Pezzo $5$-tuple: see lemma \ref{crucial}. We restart from $Y$ assuming that $\mathbb P^4$ is 
\begin{equation}
\label{DP I_S} \mathbb H := \vert \mathcal I_Y(2) \vert
\end{equation}
and $h$ is the Del Pezzo $5$-tuple considered in lemma \ref{crucial}. More precisely  $\mathbb H$ is a $4$-space of quadrics of $\mathbb P^5$ and the locus of its quadrics of rank $\leq 4$ is the union
of five nets of quadrics. These planes of $\mathbb H$ are the elements of $h$. From now on we fix the notation   \begin{equation} \label {G_H} \mathbb G_{\mathbb H} \  \text{ and} \ \mathbb G_{{\mathbb H}^*} \end{equation} respectively for the Grassmannians of planes and of lines of $\mathbb H$ in their Pl\"ucker spaces. 
At first we want to describe the discriminant sextic hypersurface of $\mathbb H$, that is the scheme of the singular quadrics $Q \in \mathbb H$. Omitting the most standard steps, let us summarize this description as follows. Consider the correspondence
$$
\tilde \Delta := \lbrace (z, Q) \in Y \times \mathbb H \ \vert \ z \in \Sing Q \rbrace,
$$
together with its natural projection maps
$$
\begin{CD}
{Y}Ê@<{q_1}<< {\tilde \Delta } @>{q_2}>> {\mathbb H} \\
\end{CD}
$$
and notice that $q_1: \tilde \Delta \to Y$ is a $\mathbb P^1$-bundle. Indeed, any projection $\pi_z: Y \to \mathbb P^4$ from a point $z \in Y$ defines an integral complete intersection of two quadric hypersurfaces
\begin{equation}
\label {Y_u} Y_z := \pi_z(Y) \subset \mathbb P^4.
\end{equation}
The pencil of quadrics through $Y_z$ pulls back to a pencil  of quadrics
 \begin{equation}
 L_z   \subset \mathbb H \label{pencil}
  \end{equation}
singular at $z$. It turns out that $q_1$ is a $\mathbb P^1$-bundle such that $q_1^*(z) =  \lbrace z \rbrace \times L_z$.
  \begin{lemm}   Assume $Q \in \mathbb H$ is singular, then $\Sing Q \cap Y \neq \emptyset$.  \end{lemm}
\begin{proof} The projection from $\Sing Q$ defines a rational map $f: Y \to \mathbb P^r$ so that $f(Y) \subset \overline Q$, where $\overline Q$ is a smooth quadric and $r \leq 4$. Let
$\Sing Q \cap Y = \emptyset$ then $f$ is a morphism. But then $f(Y)$ is a quintic surface in $\overline Q$, which is impossible. Hence $\Sing Q \cap Y \neq \emptyset$.  \end{proof}
The lemma implies the irreducibility of the closed set
\begin{equation}
\label{Delta} \Delta := q_2(\tilde \Delta).
\end{equation}
Since a general $Q \in \mathbb H$ is smooth then $\Delta$ is a hypersurface and the support of the sextic discriminant of $\mathbb H$. The name of $\Delta$ is well known, see \cite[8.5]{Do}.
Before of coming to it we recall more on its geometry, which is determined by $Y$.    
The surface $Y$ has  exactly five pencils of conics. Each of these defines a distinct Segre embedding of $\mathbb P^1 \times \mathbb P^2$ in $\mathbb P^5$, let us say \begin{equation} \label{Segre}  \Sigma_i \subset \mathbb P^5, \ i = 1 \dots 5. \end{equation} $\Sigma_i$ is union of the supporting planes of the conics of a pencil. As is well known \begin{equation} Y =  \Sigma_1 \cap \dots \cap \Sigma_5 \subset \mathbb P^5.
\end{equation} Let $\mathcal I_i$ be the ideal sheaf of $\Sigma_i$. Notice also that $ \vert \mathcal I_i(2) \vert $ is the net of quadrics 
\begin{equation}
P_i :=  \lbrace Q \in \mathbb{H} \ \vert \  \Sing Q \ \text { is a line} \  \mathbb P^1 \times \lbrace t \rbrace \subset  \mathbb P^1 \times \mathbb P^2 = \Sigma_i \rbrace. \end{equation}    Therefore $P_1 \dots P_5$ are planes in $\Delta$. Now consider in $\mathbb G_{\mathbb H}$ the corresponding set  
\begin{equation}
h = \lbrace h_1 \dots h_5 \rbrace \subset \mathbb G_{\mathbb H}
\end{equation}
of five points and, in the Pl\"ucker space of $\mathbb G_{\mathbb H}$,  the hyperplane $H_i$ such that
$$
H_i \cap \mathbb G_{\mathbb H^*} = \lbrace L \in \mathbb G_{\mathbb H^*} \ \vert \ L \cap P_i \neq \emptyset \rbrace.
$$
We observe that the previous $\mathbb P^1$-bundle $q_1: \tilde \Delta \to Y$ defines a morphism
\begin{equation}
\iota: Y \to \mathbb G_{\mathbb H^*}
\end{equation}
sending $z$ to the parameter point of the pencil $L_z$. We point out the following:
\begin{lemm}  $P_i \cap L_z$ is not empty, so that $\iota(Y) \subset H_1 \cap \dots \cap H_5$. \end{lemm}
\begin{proof} Let $\pi_z: \Sigma_i \to \mathbb P^4$ be the projection from $z$. Since $z \in \Sigma_i$ and $\Sigma_i$ is smooth of degree $3$ then $\overline Q_i := \pi_z(\Sigma_i)$ is a quadric in $\mathbb P^4$. Let $Q_i$ be its pull-back by $\pi_z$, then $Q_i$ is singular at $z$ and contains $\Sigma_i$. Hence $Q_i \in P_i \cap L_z$ and $\iota(z) \in H_1 \cap \dots \cap H_5$. \end{proof}
Let $\langle h \rangle^{\perp}$ be the orthogonal of the linear span $\langle h \rangle$ in the Pl\"ucker space of $\mathbb G_{\mathbb H^*}$, then
\begin{equation}
\label{i(Y)} \langle h \rangle^{\perp} = H_1 \cap \dots \cap H_5.
\end{equation}
Then we can consider $\iota$ as a morphism $\iota: Y \to \langle h \rangle^{\perp}$ with image in $\langle h \rangle^{\perp} \cdot \mathbb G_{\mathbb H^*}$.
 The next statement, essentially well known, will be also useful in the next sections.   
\begin{prop} \label{crucial} $\iota: Y \to \langle h \rangle^{\perp}$ is the anticanonical embedding of $Y$ and $h$ is a Del Pezzo $5$-tuple of $\mathbb H$.
Moreover $\iota(Y)$ is a linear section  of the Grassmannian $\mathbb G_{\mathbb H^*}$. \end{prop}
\begin{proof} Consider the Euler sequence of $\mathbb P^5$ restricted to $Y$
$$
0 \to \mathcal O_Y(-1) \to H^0(\mathcal O_{\mathbb P^5}(1))^* \otimes \mathcal O_Y \to T_{\mathbb P^5 \vert Y}(-1) \to0.
$$
Then its dual defines a monomorphism $\upsilon: \mathcal S^*_{\vert Y} \to H^0(\mathcal O_{\mathbb P^5}(2)) \otimes \mathcal O_Y$, where we have put $\mathcal S := \Sym^2 T_{\mathbb P^5}(-1)$.
Let $z \in Y$, then $\mathcal S^*_{\vert Y,z}$ is the vector space of quadratic forms singular at $z$ and $\upsilon_z: \mathcal S^*_{\vert Y,z} \to H^0(\mathcal O_{\mathbb P^5}(2))$ is the inclusion map.
Now let $\mathcal U$ be the pull-back by $\iota$ of the universal bundle of $\mathbb G_{\mathbb H^*}$, observe that $\mathbb P(\mathcal U) = \tilde \Delta$ and that  
$$
\begin{CD}
{\mathbb P(\mathcal U)} @>{q_2}>> {\vert \mathcal I_Y(2) \vert \times Y} \\
@VVV @VVV \\
{\mathbb P(\mathcal S^*_{\vert Y})} @>{\upsilon}>> {\vert \mathcal O_{\mathbb P^5}(2) \vert \times Y} \\
\end{CD}
$$
is a Cartesian square where the vertical maps are inclusions. The properties of the rank two vector bundle $\mathcal U$ are well known: $\mathcal O_{\tilde \Delta}(1)$ is the anticanonical sheaf.
Moreover the map $\iota: Y \to \mathbb G_{\mathbb H^*}$, defined by $\mathcal U$, is the anticanonical embedding of $Y$, followed by the inclusion $\iota(Y) \subset \mathbb G_{\mathbb H^*}$ as a linear
section. This implies the statement. \end{proof}  
\begin{defi} \it \label{Y_h} The Del Pezzo surface defined by $h$ is $Y_h := \langle h \rangle^{\perp} \cdot \mathbb G_{\mathbb H^*}.$
\end{defi}
Finally we go back to the hypersurface $\Delta$. 
 \begin{thm} \label{Segre}  $\Delta$ is the Segre primal and $2\Delta$ is the sextic discriminant of $\mathbb H$. \end{thm}
\begin{proof} In the Chow ring of the Grassmannian of lines of a $4$-space a $2$-dimensional linear section has class $(2,3)$. Since $q_2: \tilde \Delta \to \Delta$ is a birational morphism
it follows $\deg \Delta = 3$. Since $Y_h$ is smooth, it is well known that $\Delta$ is the Segre cubic primal. \end{proof}
\begin{rem} As remarked the planes $P_1 \dots P_5$ are in $\Delta$. It is easy to see that a unique quadric $Q_{ij}$  satisfies $P_i \cap P_j = \lbrace Q_{ij} \rbrace$,  $i < j$. This implies that $Q_{ij} \in \Sing \Delta$ and describes the ten singular points of $\Delta$. Notice also that $\Sing Q_{ij}$ is one of the ten lines in
$Y$ and that the obvious map $ \lbrace Q_{ij}, \  i < j  \rbrace \to \lbrace \text{lines of $Y$} \rbrace$ is bijective.
\end{rem}
\begin{rem} The previous statement has somehow a classical flavor, however we are not aware of any reference for it. We thank Igor Dolgachev for his useful comments. \end{rem}

 \section{Morin-Del Pezzo configurations}
Now we use $Y$ and the natural Del Pezzo $5$-tuple of planes $P_1 \dots P_5 \subset \mathbb H$ to describe an interesting family of special Morin configurations. We fix a linear embedding
\begin{equation} \mathbb H \subset \mathbb P(W), \end{equation}
the choice of it is irrelevant up to $\Aut \mathbb P(W)$.  We fix the notation $ W_Y := H^0(\mathcal I_Y(2))$  so that it follows 
$ \mathbb P(\wedge^3 W_Y) \subset \mathbb P(\wedge^3 W)$ and  $\mathbb G_{\mathbb H} \subset \mathbb G$. We will also assume that
 \begin{equation} u \in \lbrace h_1 \dots h_5 \rbrace = h. \end{equation}
   
\begin{defi} \label {DP marked} \it A subspace $A \subset \wedge^3 W$ is Del Pezzo marked if $\langle h \rangle \subset\mathbb P( A)$. \end{defi}
The space $\wedge^3 W_Y$ is obviously isotropic. We recall that a Morin configuration $F \subset \mathbb G$ is, by definition, a configuration of incident planes which is finite and complete. As we know, this is equivalent to say that $F$ is finite
and, moreover, that there exists a maximal isotropic space $A \in LG(10, \wedge^3 W)$ such that $F = \mathbb P(A) \cdot \mathbb G$ and $\langle F \rangle = \mathbb P(A)$.
\begin{defi} \it \label{DP}Let $F$ be a Morin configuration: we say that $F$ is a Morin-Del Pezzo configuration if $F$ contains $h$ and $\langle h \rangle \cdot \mathbb G = h$. \end{defi}
Let us point out that $h \subset F$ implies $A \cap \wedge^3 W_Y = [h]$, that is,
\begin{equation}
\label{[h] cap A}  \langle h \rangle = \mathbb P(A) \cap \mathbb P(\wedge^3 W_Y).
\end{equation}
This follows because, counting dimensions, the intersection $L \cap \mathbb G_{\mathbb H}$ is not finite for any space $L \subset \mathbb P(\wedge^3 W_Y)$ which contains $\langle h \rangle$ properly.  
Let
\begin{equation}
\label {F'} F' := \mathbb P(A) \cdot  (\mathbb P(\wedge^3W) - \mathbb P(\wedge^3W_Y)),
\end{equation}
then the condition $\langle h \rangle \cdot \mathbb G = h$ is equivalent to say that
\begin{equation}
  F = F' \cup h.
\end{equation}
We fix the notation $F'$ for the subscheme of $F$ occurring in this decomposition.
\begin{rem}  In this part of the paper we describe Morin-Del Pezzo configurations and give a method for their explicit construction in any possible length.  As we will see, these configurations are strictly related
to the family of $V$-threefolds containing a plane and to the Severi variety of quadratic sections $C$ of $Y$ such that $\Sing C$ has length $\geq 6$. \par We stress however that our construction only gives
Morin configurations of special type. The reason is that $h$ spans a $3$-space. Since any Morin configuration spans a $9$-space, otherwise it is not complete, it follows that  the length of 
$F \cup h$ is \it at least \rm $11$, while a general configuration has length $10$. Nevertheless this construction recovers most families of Morin configurations for any length 
$k \in [11 , 20]$.  As we will see, the family of Morin-Del Pezzo configurations is irreducible and depend on $9$ moduli.
\end{rem}
\medskip  \par
To begin let us fix since now a vector $f \notin W_Y$ and the decomposition
\begin{equation}
W = F \oplus W_Y,
\end{equation}
where $F$ is generated by $f$. Moreover we fix the identification
$$ 
\wedge^2 W_Y = \lbrace f \wedge b, \ b \in \wedge^2 W \rbrace
$$
and the decomposition $ \wedge^3 W = \wedge^3 W_Y \oplus \wedge^2 W_Y$. So far we then have
\begin{equation}
\label {r} \wedge^3 (W_Y \oplus  F) = \wedge^3 W = \wedge^3 W_Y \oplus \wedge^2 W_Y.
\end{equation}
In particular any two vectors $v, v' \in \wedge^3 W$ are uniquely decomposed as $ v = a + f \wedge b $ and  $v' = a' + f \wedge b'$, where $a, a' \in \wedge^3 W_Y$ and $b, b' \in \wedge^2 W_Y$. Therefore we have
\begin{equation}
\label {w DP} w(v,v') = v \wedge v' = -f \wedge (a \wedge b' + a' \wedge b).
\end{equation}
Notice that $w$ is induced by the natural pairing $\wedge^3W_Y \times \wedge^2 W_Y \to \wedge^5 W_Y$, up to a non zero factor the choice of $f$ is irrelevant. The proof of the next lemma is immediate.
  \begin{lemm} The subspaces $\wedge^3 W_Y$ and $\wedge^2 W_Y$ are isotropic spaces of $w$. \end{lemm}
Let $r: \wedge^3 W \to \wedge^2 W_Y$ be the map sending $a + f \wedge b$ to $b$, then $r$ has a geometric meaning. Indeed $r$ defines the projection of center $\mathbb P(\wedge^2 W_Y)$ 
 \begin{equation}\label {overline r} \overline r: \mathbb P(\wedge^3 W) \to \mathbb P(\wedge^2 W_Y). \end{equation}   
 Now let $\mathbb G_{\mathbb H^*} \subset \mathbb P(\wedge^2 W_Y)$ be the Grassmannian of lines of $\mathbb H$, then we have:
   \begin{lemm} Let $o \in \mathbb G$,  the assignement $o \to P_o \cap \mathbb H$ defines the rational map $$ \overline r \vert \mathbb G: \mathbb G \to \mathbb G_{\mathbb H^*} \subset \mathbb P(\wedge^2 W_Y).$$
  \end{lemm}
   \begin{proof} Let $v \in \wedge^3 W - \wedge^3W_Y$ be decomposable and defining $o$. Then we have $v = b \wedge f'$ with $b$ decomposable in $\wedge^2 W_Y$ and $f' \in W - W_Y$. We can write $f'$ as $f' = kf + c$ with $c \in W_Y$. Then $v = a + kf \wedge b$ with $a = b \wedge c$. Hence $P_o \cap \mathbb H$ is the line defined by the vector 
$-kb = r(v)$ and the statement follows. \end{proof}
   \begin{rem} \label{resolution} In particular the fibre of $\overline r \vert \mathbb G$ at $\overline r(u)$ is the $\mathbb P^3$ of planes of $\mathbb I$ containing the line $P_o \cap \mathbb H$ and the next commutative diagram  solves the indeterminacy of $\overline r \vert \mathbb G$:
\begin{equation} 
\xymatrix{ & & & {\tilde{ \mathbb G} } \ar[dl]_{\gamma } \ar[dr]^{\tilde r} & \\
 & &   { \mathbb G}  \ar@{>}[rr]^{\overline r \vert \mathbb G} & &   {\mathbb G_{\mathbb H^*} }.  }
 \end{equation}
In it $\tilde {\mathbb G}$ is the correspondence  defined below and $\gamma$, $\tilde r$ are its projections.  $\tilde r$ is a $\mathbb P^3$-bundle. 
\begin{equation}
  \tilde {\mathbb G} := \lbrace (L, P) \in \mathbb G_{\mathbb H^*} \times \mathbb G \ \vert \ L \subset P \rbrace.
\end{equation} 
 \end{rem} 
Now we consider the family of all isotropic spaces $A$ in $\wedge^3 W$ which are marked by the Del Pezzo $5$-tuple  $h$, that is such that $ h = \lbrace h_1 \dots h_5 \rbrace \subset \mathbb P(A)$.
Let $i = 1 \dots 5$ and let $s_i \in \wedge^3W$ be a vector defining the point $h_i$, then we have the orthogonal space
 \begin{equation}
\label {H^ort} H^{ort} := \lbrace s_1 \dots s_5 \rbrace^{\perp} \subset \wedge^3 W.
\end{equation}
Since $s_1 \dots s_5$ generate a subspace of dimension $4$ it follows $\dim H^{ort} = 16$. Let
\begin{equation} H_Y := r(H^{ort}) \subset \wedge^2 W_Y,
\end{equation}
since $\Ker r = \wedge^3 W_Y$ we have the exact sequence of vector spaces
\begin{equation}
\label{H_Y} 0 \to \wedge^3 W_Y \to H^{ort} \stackrel{r} \to H_Y \to 0. 
\end{equation}
Let $\wedge^3 W_Y \times \wedge^2 W_Y \to \wedge^5 W_Y$ be the natural pairing. It follows from the geometric description of $r$ in \ref{overline r} that $H_Y$ is the orthogonal of  $\lbrace s_1 \dots s_5 \rbrace$ under such a pairing.  
Let $h \subset A$, where $A \subset \wedge^3 W$ is an \it isotropic \rm subspace, then we have $A \subset H^{ort}$ and $$r(A) \subseteq H_Y. $$ Furthermore, under the previous pairing, we have the  equality:
\begin{equation}
\label{equal} r(A) = (A \cap \wedge^3 W_Y)^{\perp}.   
\end{equation}
Then, since $H_Y$ is $6$-dimensional, the next lemma follows.
\begin{lemm} Let $A$ be maximal isotropic then $r(A) = H_Y$ iff $A \cap \wedge^3 W_Y = [h]$. \end{lemm}
Let $H_1 \dots H_5 \subset \mathbb P(\wedge^2 W_Y)$ be the hyperplanes respectively defined by $h_1 \dots h_5$. As in \ref{Y_h},  $\mathbb P(H_Y)$ is the $5$-space spanned by the smooth Del Pezzo quintic surface
\begin{equation}
 Y_h = H_1 \cap \dots \cap H_5 \cap \mathbb G_{\mathbb H^*}.
\end{equation}
\par
Now assume that $F' := (\mathbb P(A) - \langle h \rangle)  \cdot \mathbb G$ is finite, then we have:
\begin{lemm} \label{embed} $\overline r$ restricted to $F'$ is an embedding. \end{lemm}
\begin{proof} Let $\zeta \subset F'$ be a scheme of length $2$ such that $\overline r \vert \zeta$ is not an embedding. Then the line $\langle \zeta \rangle$ intersects $\langle h \rangle$ and
$\zeta$ is contained in a fibre of $\overline r$. This, by remark \ref{resolution},  is a $3$-space linearly embedded in $\mathbb G$. It is the family of planes containing
a fixed line of $\mathbb H$. But then $\langle \zeta \rangle$ is a pencil of planes contained in $F'$ and $F'$ is not finite: a contradiction. \end{proof}
 Now we concentrate on Morin-Del Pezzo configurations. We start more in general from a maximal isotropic $A$.  Keeping our notation we assume
$$
F = \mathbb P(A) \cdot \mathbb G = h  \cup   F'
$$
where $\langle h \rangle \cdot \mathbb G = h $ and $F'$ is finite. Let $V_A$ be the $V$-threefold defined by $A$, we want to reconstruct it explicitly and see that it is rational. Notice that
 $u \in h$. We put $u = h_5$ and consider the projection map  from which $V_A$ is constructed. We know that this is the restriction to $\mathbb P(A)$ of the tangential projection of $\mathbb P(\wedge^3W)$ from the embedded tangent space to $\mathbb G$ at $u$. This is just the projection from $u$, we denote since now as
\begin{equation}
p: \mathbb P(A) \to \mathbb P^8,
\end{equation}
see \ref{tang}. $\mathbb P^8$ is the space of the Segre embedding $\mathbb P^2 \times \mathbb P^2$ of $P_u \times P_u^{\perp}$ and we know that  
$$
V_A \subset \mathbb P^2 \times \mathbb P^2 \subset \mathbb P^8.
$$
Since $h$ spans a $3$-space containing $u$ we can add to our play the plane
\begin{equation}
\label {plane}  \it P_h := p(\langle h \rangle).
\end{equation}
Moreover we will also consider the set of four points
\begin{equation}
\label {4 points} h_u := \lbrace p_u(h_1) \dots p_u(h_4) \rbrace \subset P_h.
\end{equation}
These are in general position in $P_h$, since the same is true in $\langle h \rangle$ for $h$.  
 \begin{thm}$V_A$ contains the plane $P_h$, in particular $V_A$ is rational.  \end{thm}
\begin{proof}  We know that $V_A$ has bidegree $(2,2)$ and isolated singularities. Now assume that $P_h \subset \mathbb P^2 \times \mathbb P^2$. Then, since $V_A$ is singular at the four points of $h_u$, it is clear that $P_h \cdot V_A$ cannot be a conic.  This implies that $P_h \subset V_A$. Hence it suffices to show that $$ P_h \subset \mathbb P^2 \times \mathbb P^2. $$  Assume $P_h$ is not in $\mathbb P^2 \times \mathbb P^2$ and consider the scheme $D := P_h \cdot \mathbb (P^2 \times \mathbb P^2)$. Then $D$ contains the set $h_u$ of four points in general position but  $D \neq P_h$. We claim that then $D$ is a conic. Let us prove this fact: the variety $\Sigma$ of bisecant lines to $\mathbb P^2 \times \mathbb P^2$ is a well known cubic hypersurface and a Severi variety. In particular $\Sigma$ contains the six lines joining two by two the points of $h_u$. Hence $P_h$ is in $\Sigma$, though not in $\mathbb P^2 \times \mathbb P^2$. It is known that every such a plane cuts exactly a conic of bidegree $(1,1)$ on $\mathbb P^2 \times \mathbb P^2$, cfr. \cite[chapter 5]{Ru}. Then $D$ is a conic and its projections in the factors are lines $L_1$ and $L_2$. We have $D \subset L_1 \times L_2$ and $L_1 \times L_2$ is embedded in $\mathbb P^2 \times \mathbb P^2$ as a quadric. Assume $L_1 \times L_2$ is not in $V_A$ then $V_A \cdot (L_1 \times L_2)$ is a quadratic section of $L_1 \times L_2$, singular at the set of coplanar points $h_u$. This implies 
$V_A \cdot ( L_1 \times L_2) = 2D$. Notice also that $D$ spans $P_h$.  \par Now we can fix coordinates $(x_1:x_2:x_3) \times (y_1:y_2:y_3)$ on $\mathbb P^2 \times \mathbb P^2$ so that
\begin{equation}  L_1 \times L_2 = \lbrace x_3 = y_3 = 0 \rbrace \ \text{ and } \ D = \lbrace x_3 = y_3 = d = 0 \rbrace, \end{equation} $d$ being a form of bidegree $(1,1)$  in $(x_1:x_2) \times (y_1:y_2)$. 
Then $2D$ is the complete intersection $\lbrace x_3 = y_3 = d^2 = 0 \rbrace$ and the equation of $V_A$ is $ ax_3 + by_3 + kd^2 = 0$, where $a$ and $b$ are forms of bidegrees $(1,2)$ and $(2,1)$ and $k \neq 0$.  If $L_1 \times L_2$ is in $V_A$ we have $k = 0$. One computes that $\Sing V_A \cap D$ is defined by the equations $a = b = d = x_3 = y_3 = 0$. Moreover $a, b, d$ define in $L_1 \times L_2$ curves $C_a, C_b, D$ of bidegrees $(1,2), (2,1), (1,1)$ and $\Sing V_A$ is finite. Since $C_aD = C_bD = 3$, it follows that $\Sing V_A \cap D$ contains at most three singular points of $V_A$. Since $h_u$ has cardinality $4$ this is a contradiction. Hence we can conclude that $P_h \subset V_A$. Finally, the rationality of $V_A$ follows from the explicit birational map $V_A \to \mathbb P^1 \times \mathbb P^2$ we construct in the next section.
 \end{proof} 
\begin{rem}  Let $F = \mathbb P(A) \cdot \mathbb G$ be any Morin configuration, smooth at $u$ as we assume in this paper, and $V_A$ its associated $V$-threefold. If
 $F$ has length $\geq 16$ then theorem \ref{no plane} implies that $V_A$ contains a plane. Up to $\Aut \mathbb P^2 \times \mathbb P^2$ we can assume that this is $P_h$. Thus Morin configurations of length $\geq 16$ are basically Morin-Del Pezzo configurations.
 \end{rem}
 \section{The $V$-threefold of a Morin-Del Pezzo configuration}
 In this section we construct $V$-threefolds associated to Morin-Del Pezzo configurations.
  Let $F = \mathbb P(A) \cdot \mathbb G$ be a Morin-Del Pezzo configuration and let
  \begin{equation}
  \label{plane} P_h := \lbrace o \rbrace \times \mathbb P^2
  \end{equation}
 be the plane contained in the threefold $V_A$. Now we consider the projection  
   \begin{equation} \label {p_h}
 p_h: \mathbb P^8 \to \mathbb P^5.
\end{equation}
of $\mathbb P^8$ from $P_h$ and study $p_h \vert V_A$. Let us point out that $p_h$ factors as in the diagram
\begin{equation} \label{factor p_h}
\xymatrix{ & & & { \mathbb P(A) } \ar[dl]_{p} \ar[dr]^{\overline r_{ \vert \mathbb P(A)}} & \\
 & &   { \mathbb P^8}  \ar@{>}[rr]^{p_h} & &   {\mathbb P^5 },  }
 \end{equation}
where $\overline r$ is as in section 6. Indeed, $\overline r_{\vert \mathbb P(A)}$ is the projection
from $\langle h \rangle$, while $p$ and $p_h$ are the projections from $u$ and $\langle h_u \rangle$. Then, since $h_u = p(h)$, it follows $\overline r_{\vert \mathbb P(A)} = p_h \circ p$. 
\begin{rem} \label{ triangle} Notice that $\mathbb P^5 = \mathbb P(H_Y) \subset \mathbb P(\wedge^2 W_Y)$ and that $\mathbb P(H_Y) \cdot \mathbb G_{\mathbb H^*}$
is the quintic Del Pezzo surface defined by $<h>^{\perp}$. This, by the definition of $\overline r$,  is the locus \begin{equation} Y_h =  \lbrace \ell_o, o \in Y \rbrace, \end{equation} where $\ell_o$ is the pencil of quadrics of $\mathbb H$ singular at $o$. See \ref{H_Y} and also  lemma \ref{embed}. \end{rem} Let $\sigma: \mathbb P \to \mathbb P^2 \times \mathbb P^2$ be the blowing of $P_h$ then we have the commutative diagram 
\begin{equation} 
\begin{CD}
{\mathbb P} @>{\sigma}>> {\mathbb P^2 \times \mathbb P^2} \\
@VV{\tilde p_h}V @VV{{p_h}_{\vert \mathbb P^2 \times \mathbb P^2}}V \\
{\mathbb P^1 \times \mathbb P^2} @>>> {\mathbb P^5}, \\
\end{CD}
\end{equation}
where $\tilde p_h$ is a $\mathbb P^1$-bundle and the bottom arrow is the Segre embedding. Let us consider the projection $p_o: \mathbb P^2 \to \mathbb P^1$ from the point $o$, then we have
  \begin{equation}
\label{p_o} \tilde p_h \circ \sigma^{-1} = p_o \times id_{\mathbb P^2}: \mathbb P^2 \times \mathbb P^2 \to \mathbb P^1 \times \mathbb P^2.
\end{equation}
Moreover, let $E \subset \mathbb P$ be the exceptional divisor of $\sigma$. Since $P_h$ has trivial normal bundle the morphism $\tilde p_h: E \to \mathbb P^1 \times \mathbb P^2$ is biregular and its inverse defines 
a regular section
\begin{equation} \label{ex.div.}
s: \mathbb P^1 \times \mathbb P^2 \longrightarrow E \subset \mathbb P.
\end{equation}
We want to study the diagram more in detail with respect to $V_A$. Denoting by $\tilde V_A$ the strict transform of $V_A$ via $\sigma$, and by $p_A$ the restriction of $p_h$ to $V_A$, we have: 
\begin{equation}
\label{crucial CD bis}
\begin{CD}
{\tilde V_A} @>{\sigma}>> {V_A} \\
@VV{\tilde p_A}V @VV{p_h}V \\
{\mathbb P^1 \times \mathbb P^2} @>>> {\mathbb P^5}, \\
\end{CD}
\end{equation}
with $\tilde p_A = \tilde p_h \vert \tilde V_A$. It is clear that $\tilde V_A$ is rational, because it is an integral member of
\begin{equation}
\label{tilde H} \vert  \mathcal O_{\mathbb P}(2\tilde H - E) \vert
\end{equation}
where $\mathcal O_{\mathbb P}(\tilde H) \cong \sigma^* \mathcal O_{\mathbb P^2 \times \mathbb P^2}(1,1)$. Since $2\tilde H - E$ has degree one on the fibres of $p_h$ 
then
$$
\tilde p_A: \tilde V_A \longrightarrow \mathbb P^1 \times \mathbb P^2
$$
is birational. Let $E_h = E \cdot \tilde V_A$ be the strict transform of $P_h$ by $\sigma_{\vert \tilde V_A}$, then
\begin{equation}
\label{E_h} \sigma_{\vert V_A}: \tilde V_A-E_h \longrightarrow V_A - P_h
\end{equation}
 is a biregular map. 
 \begin{lemm} $\sigma_{\vert E_h}: E_h \to P_h$ is the blowing up of $h_u$ and $E_h$ is a smooth quintic Del Pezzo surface.
 \end{lemm}
 \begin{proof} We have $E = \mathbb P^1 \times P_h$ so that $\sigma_{\vert E}: E \to P_h$ is the natural projection. Let us compute the bidegree
 $(m,n)$ of $E_h$ in $\mathbb P^1 \times E_h$. Since $\tilde V_A$ has degree one on the fibres of $\tilde p_h$, it follows $m = 1$. Now notice 
 that $\Sing V_A \cdot P_h = h_u$, because $\langle h \rangle \cdot \mathbb G = h$.  Then, writing a local equation for a $V$-threefold containing 
 a plane like $P_h$, it is easy to deduce that the pencil of conics through $h_u$ lifts, by $\sigma_{\vert E_h}$, to a pencil of conics. This is cut by the 
 ruling of planes of $E$. Hence $n = 2$ and $E_h$ is the blowing up of $P_h$. Since $h_u$ is a set of four points in general position, then $E_h$ is a 
 smooth quintic Del Pezzo surface.  \end{proof}  Notice also that $\mathcal O_{E_h}(1,1) \cong \omega_{E_h}^{-1}$, therefore the Segre embedding of $E$ restricts to the anticanonical
map of $E_h$. Moreover the next theorem follows.
\begin{thm} $\sigma: \tilde V_A \to V_A$ is the small contraction of four disjoint copies of $\mathbb P^1$. \end{thm}
Let us fix the notation $Y_h := \tilde p_A(E_h)$. This is a smooth quintic Del Pezzo surface  
\begin{equation}
Y_h \subset \mathbb P^1 \times \mathbb P^2 \subset \mathbb P^5.
\end{equation}
Now we describe the birational morphism $\tilde p_A: \tilde V_A \to \mathbb P^1 \times \mathbb P^2$ in order to invert it. To this purpose it is useful to consider the
conic bundle $\pi: V_A \to \mathbb P^2$ defined by the projection of $\mathbb P^2 \times \mathbb P^2$ onto the second factor. We have the commutative diagram
\begin{equation}
\label{CD Y_h}
\begin{CD}
{E_h} @>>> {\tilde V_A} @>{\sigma}>> {V_A} \\
@VV{\tilde p_A\vert E_h}V @VV{\tilde p_A}V @VV{\pi}V \\
{Y_h} @>>> {\mathbb P^1 \times \mathbb P^2} @>{\tilde \pi}>> {\mathbb P^2} \\
\end{CD}
\end{equation}
where $\tilde \pi$ is the projection map. Indeed, let $t \in \mathbb \PP^2$ then $\pi^*(t)$ is $V_A \cdot \mathbb( \PP^2 \times \lbrace t \rbrace)$ and
$P_h \cdot \pi^*(t) = (o,t)$. Moreover, $\tilde p_A \circ \sigma^{-1} \vert \pi^*(t)$ is precisely the projection from $(o,t)$
$$ p_{o,t}: \pi^*(t) \to \mathbb P^1 \times \lbrace t \rbrace. $$
Notice that $(o,t) \in \pi^*(t) \subset P_h = \lbrace o \rbrace \times \mathbb P^2$. It is clear that the tangent space to $V_A$ at $(o,t)$ has dimension $4$ if
$(o,t) \in \Sing \pi^*(t)$. This implies the next lemma.
\begin{lemm} Assume $(o,t) \in P_h - h_u$ then $\pi^*(t)$ is smooth at $(o,t)$. \end{lemm}
Let $\Gamma \subset \mathbb P^2$ be the discriminant sextic of $\pi$ and $t \in \mathbb P^2 - \Gamma$, then $\pi^*(t)$ is a smooth conic. Let $\pi^*(t)'$ be
its strict transform by $\sigma$, then $\pi^*(t)' = (\tilde p_A \circ \tilde \pi)^*(t)$ and
$$ \tilde p_A \vert \pi^*(t)': \pi^*(t)' \to \mathbb P^1 \times \lbrace t \rbrace $$
is biregular and induced by $p_{o,t}$. Moreover, $\tilde p_A$ is regular on $\pi^*(t)'$. In  $E$ we define:
\begin{equation}
\label {model C_h} C_h = E \cdot (\pi \circ \sigma)^*\Gamma.
\end{equation}
$C_h$ is a curve embedded in the Del Pezzo surface $E_h = E \cdot \tilde V_A$. Let 
\begin{equation}
\label{s_o} s_o: \mathbb P^2 \to V_A
\end{equation}
 be the linear isomorphism such that $s_o(t) = (o,t)$ and let $\Gamma_h := s_o(\Gamma)$. Then the next lemma is standard, we omit further details.
 \begin{lemm} $C_h$ is the strict transform of $\Gamma_h$ by the blowing up $\sigma_{\vert E_h}: E_h \to P_h$.  In particular $C_h$ is a quadratic section of the anticanonical embedding of $E_h$.  
 \end{lemm}
 Finally let us define and consider the following varieties
\begin{defi} \label { C and F} $C := p_{h*}C_h$ and $F := p_h^*C.$ \end{defi}
$F$ is a $\mathbb P^1$-bundle over $C$ and  $C$ is the biregular to $C_h$ via $p_h$. We have
$$
C \subset Y_h \subset \mathbb P^1 \times \mathbb P^2 \subset \mathbb P^5.
$$ 
We recall that $C$ is complete intersection in $\mathbb P^1 \times \mathbb P^2$ of $Y_h$ and a quadratic section.  
 \begin{thm} $\tilde p_A: \tilde V_A \to \mathbb P^1 \times \mathbb P^2$ is the contraction of $F$. \end{thm}
\begin{proof} Let $\zeta \subset \tilde V_A$ be a scheme of length $2$. Assume that the morphism $\tilde p_A$ is not an
embedding on $\zeta$. Then $\zeta \subset f$ for a fibre $f$ of $p_h: \mathbb P \to \mathbb P^1 \times \mathbb P^2$.
Since $\tilde V_A$ has intersection index $1$ with $f$, it follows $f \subset \tilde V_A$. Notice also that, as every fibre of $p_h$,  $\sigma_*f$ is a line in
a plane $\mathbb P^2 \times \lbrace t \rbrace$. Hence the fibre $\pi^*(t)$ cannot be a smooth conic, since it  contains the line $f$. Then we
have $\sigma_*f \subset \pi^* \Gamma$ and $f \subset F$.  This implies the statement. \end{proof}
\begin{rem} In a more descriptive way let $t$ be a point of $\Gamma$ such that $t \notin h_u$. Then $\pi^*(t)$ is a rank $2$ conic and it is not singular at
$(o,t)$, as remarked. Let $f + \overline f \subset \tilde V_A$ be its strict transform by $\sigma$. Then a summand, say $f$, is a fibre of $p_h$ and intersects
$F_h$. For the other summand the map  $\tilde p_A: \overline f \to \mathbb P^1 \times \lbrace t \rbrace$ is a linear isomorphism.  \end{rem}
\begin{rem} \label{Wirtinger} Let $g: \tilde {\Gamma} \to \Gamma$ be the degree $2$ cover defined by $\pi: V_A \to \mathbb P^2$. Then
$\tilde \Gamma$ parametrizes the lines in $\pi^*(t)$, $t \in \Gamma$. At a general $t$ we have $g^*(t) = \lbrace f, \overline f \rbrace$. Since 
$f$ and $\overline f$ are distinguished by the intersection with $P_h$, then $\tilde \Gamma$ is split over $\Gamma$. If $\Gamma$ is nodal one can 
see that $g$ is a Wirtinger cover of $\Gamma$, in the sense of \cite[section 5]{B2}. \end{rem}

We can now reconstruct $V_A$, describing explicitly the inverse map 
\begin{equation} \ \sigma \circ \tilde p_A^{-1}: \mathbb P^1 \times \mathbb P^2 \longrightarrow \mathbb P^8 \end{equation}
and its image $V_A \subset \mathbb P^2 \times \mathbb P^2 \subset \mathbb P^8$. Let $\mathcal J_{P_h}$ be the ideal sheaf of $P_h$ in $\mathbb P^2 \times \mathbb P^2$, then the rational map $p_h$ is defined by  $\vert \mathcal J_{P_h}(1,1) \vert$. Since we have 
\begin{equation} F_h = (\sigma \vert \tilde V_A)^* P_h \ \text{and} \ Y_h = \tilde p_{A*} F_h,
\end{equation}
 it follows 
\begin{equation}
\tilde p_{A*}(\sigma \vert \tilde V_A)^* \vert \mathcal J_{P_h}(1,1) \vert = Y_h + \vert \mathcal O_{\mathbb P^1 \times \mathbb P^2}(1,1) \vert \subset \vert \mathcal O_{\mathbb P^1 \times \mathbb P^2}(2,3) \vert
 \end{equation}
where $Y_h + \vert \mathcal O_{\mathbb P^1 \times \mathbb P^2}(1,1) \vert$ denotes the linear system of divisors of bidegree $(2,3)$ having $Y_h$ as a fixed component. This is contained in  the linear system $\mathbb J$, defining the rational map $ \sigma \circ \tilde p_A^{-1}: \mathbb P^1 \times \mathbb P^2 \to \mathbb P^8$. Let $\mathcal J_C$ be the ideal sheaf of $C$. Then we  have  \begin{equation}
\label{ideal C}
Y_h + \vert \mathcal O_{\mathbb P^1 \times \mathbb P^2}(1,1) \vert  \subset \mathbb J \subseteq  \vert \mathcal J_C(2,3) \vert,
\end{equation}
 just because $C$ is in the indeterminacy of $\sigma \circ \tilde p_A^{-1}$. Now the target space of this rational map is $\mathbb P^8$, since $V_A$ is not contained in a hyperplane. This implies   $\dim \mathbb J = 8$ and makes our reconstruction much simpler.    
   \begin{thm} $\sigma \circ \tilde p_A^{-1}: \mathbb P^1 \times \mathbb P^2 \to \mathbb P^8$ is defined by the linear system $\vert \mathcal I_C(2,3) \vert$. \end{thm}
\begin{proof} It suffices to show that $\dim \vert \mathcal J_C(2,3) \vert = 8$. This follows, with the usual notation,  from the standard exact sequence of ideal sheaves
$$
0 \to \mathcal J_{Y_h}(2,3) \to \mathcal J_C(2,3) \to \mathcal J_{C \vert Y_h}(2,3) \to 0.
$$
It is easy to see that this is actually the sequence
$$
0 \to \mathcal O_{Y_h}(1,1) \to \mathcal J_C(2,3) \to \mathcal O_{Y_h}(0,1) \to 0.
$$
Passing to the associated long exact sequence it follows $h^0(\mathcal J_C(2,3)) = 9$.
\end{proof}
 Finally let us remark that $h^0(\mathcal O_{\mathbb P^1 \times \mathbb P^2}(1,1)) = h^0(\mathcal J_C(2,2)) = 3$ and let  
\begin{equation}
\label{mu}\mu: H^0(\mathcal O_{\mathbb P^1 \times \mathbb P^2}(0,1)) \otimes H^0(\mathcal J_C(2,2)) \to H^0(\mathcal I_C(2,3))
\end{equation}
be the multiplication map. Consider the rational maps 
\begin{equation} \label{reconstruct}
\pi_1, \pi_2: \mathbb P^1 \times \mathbb P^2 \to \mathbb P^2 \times \mathbb P^2
\end{equation}
respectively defined by the net of surfaces $\vert \mathcal J_C(2,2) \vert$ and $\vert \mathcal O_{\mathbb P^1 \times \mathbb P^2}(0,1) \vert$. We claim that
$\mu$ is an isomorphism. Then $\sigma \circ \tilde p_A^{-1}: \mathbb P^1 \times \mathbb P^2 \to \mathbb P^8$ clearly factors through the product map
$\pi_1 \times \pi_2$ and the Segre embedding $\mathbb P^2 \times \mathbb P^2 \to \mathbb P^8$. This makes clear how to effectively reconstruct $V_A$
from $\pi_1 \times \pi_2$. Let us prove our claim.
\begin{thm}Ê$\mu$ is an isomorphism. \end{thm}
\begin{proof} Consider the standard exact sequence of ideal sheaves of $\mathbb P^1 \times \mathbb P^2$
$$
0 \to \mathcal J_{Y_h} (2,2) \to \mathcal J_C(2,2) \to \mathcal J_{C \vert Y_h}(2,2) \to 0.
$$
Since $Y_h$ has bidegree $(1,2)$ this is just
$$
0 \to \mathcal O_{\mathbb P^1 \times \mathbb P^2} (1,0) \to \mathcal J_C(2,2) \to \mathcal O_{Y_h} \to 0.
$$
Tensor it by $L \otimes \mathcal O_{\mathbb P^1 \times \mathbb P^2}$ with $L := H^0(\mathcal O_{\mathbb P^1 \times \mathbb P^2}(0,1))$. Passing to the corresponding
long exact sequences, one obtains the exact commutative diagram
$$
\begin{CD}
0 @>>> {L \otimes H^0(\mathcal O_{\mathbb P^1 \times \mathbb P^2}(1,0))}  @>>> {L \otimes H^0(\mathcal I_C(2,2))}Ê@>>> {L \otimes H^0(\mathcal O_{Y_A}})@>>> 0 \\
@. @VV{\mu_1}V @VV{\mu}V @VV{\mu_2}V @. \\
0 @>>> {H^0(\mathcal O_{\mathbb P^1 \times \mathbb P^2}(1,1))} @>>> {H^0(\mathcal I_C(2,2))} @>>> L @>>> 0 \\
\end{CD}
$$
where $\mu_1, \mu_2$ are multiplication maps and isomorphisms. Then $\mu$ is an isomorphism.
\end{proof}
Finally, we conclude this section by the following remark. 
\begin{rem} \label{determ} As above let $\Gamma \subset \mathbb P^2$ be the discriminant sextic of $\pi: V_A \to \mathbb P^2$. The set $\Sing \Gamma$ contains the set of four points $\pi(h_u)$. Let $\mathcal I$ be its ideal sheaf in $\mathbb P^2$, then the product map $H^0(\mathcal I(2)) \otimes H^0(\mathcal O_{\mathbb P^2}(1)) \to H^0(\mathcal I_{\mathbb P^2}(3))$ is an isomorphism. Moreover, the pencil of
conics $\vert \mathcal I(2) \vert$ defines a rational map $q: \mathbb P^2 \to \mathbb P^1$ and hence the birational embedding
$$
 q \times id_{\mathbb P^2}: \mathbb P^2 \longrightarrow  \mathbb P^1 \times\mathbb P^2 \subset \mathbb P^5,
$$
whose image in $\mathbb P^1 \times \mathbb P^2$ is $Y_h$. We know that $C$ is the strict transform of $\Gamma$ by $q \times id_{\mathbb P^2}$. Composing $q \times id_{\mathbb P^2}$ with
the product map $\pi_1 \times \pi_2$ we obtain the plane $P_h$. Moreover, the image of $\pi_1 \times \pi_2$ in $\mathbb P^2 \times \mathbb P^2$  is the $V$-threefold $V_A$ and we retrieve 
$\Gamma$ as the discriminant curve of its projection $\pi: V_A \to \mathbb P^2$. Clearly this construction always works: under the only assumption that the sextic $\Gamma$ contains four singular points in general position. This shows the next property. 
\begin{thm} Let $\Gamma \subset \mathbb P^2$ be any sextic with four singular points in general position. Then $\Gamma$ is the discriminant curve of a conic bundle
$\pi: V_{\Gamma} \to \mathbb P^2$ such that:
\begin{enumerate}
\item $V_{\Gamma}$ is a bidegree $(2,2)$ threefold in $\mathbb P^2 \times \mathbb P^2$,
\item $\pi: V_{\Gamma} \to \mathbb P^2$ is one of the two projections,
\item $V_{\Gamma}$ contains a plane $P$ transversal to $\pi$.
\end{enumerate}
\end{thm}
\end{rem}
As in remark \ref{Wirtinger},  $\pi$ defines a double cover $g: \tilde \Gamma \to \Gamma$ which is split over $\Gamma$. If $\Gamma$ is nodal $\pi$ is a Wirtinger cover. Hence $\tilde \Gamma$ is the
gluing, according to the prescriptions,  of  two copies of the partial normalization of $\Gamma$ at the above mentioned four nodes.
\section{Geometry of Morin-Del Pezzo configurations}
Now we describe the truly geometric construction of a Morin-Del Pezzo configuration like $F$.
We infer that such configurations form an irreducible family
 It turns out that $F$ is determined by the curve $C$ and $\Sing C$ as follows. Let
$\nu: C^n \to C$ be the normalization map, then $\Sing C$ is defined by the exact sequence
$$
0 \to \mathcal O_C \stackrel{\nu^*} \to \mathcal O_{C^n} \to \mathcal O_{\Sing C} \to 0
$$
as usual. Restricting to $F'$ the commutative diagram of linear maps \ref{factor p_h}, we obtain
\begin{equation} 
\xymatrix{ & & & { F' } \ar[dl]_{p_{\vert F'}} \ar[dr]^{\overline r_{ \vert F'}} & \\
 & &   { \mathbb P^8}  \ar@{>}[rr]^{p_h} & &   {\mathbb P^5. }  }
 \end{equation}
Here $\overline r_{\vert F'}$ is an embedding by lemma \ref{embed} and $p_{\vert F'}$ embeds $F'$ in $U := V_A-P_h$. Hence $F'$ is biregular to $p(F')$ and $p_h$ embeds $p(F')$ in $\mathbb P^5$. On the other hand let $R \subset V_A$ be the ramification scheme of $\pi: V_A \to \mathbb P^2$. Then $\sigma^*R$ is contained in the fundamental divisor $Z$ of $\tilde p_A: \tilde V_A \to \mathbb P^1 \times \mathbb P^2$. More precisely we have $\tilde p_A(Z) = C$ so that $\tilde p_A: Z \to C$ is a $\mathbb P^1$-bundle. Then $\sigma^*R$ is a birational section of it cutting on
$F \cdot U$ the locus of the singular points of the singular fibres of $\pi$. Then  theorem \ref{main} implies that
\begin{equation}
p(F') = U \cdot \Sing R.
\end{equation}
  Since $\sigma^{-1}: U \to \tilde V_A$ is an open embedding and $p_h \vert U = \tilde p_A \circ \sigma^{-1}\vert U$, it follows:
\begin{lemm}  The rational map $\overline r$ embeds $F'$ in $\Sing C$. \end{lemm}
Let $F_h := \overline r (F') \subset \mathbb P^5$, the next lemma will be useful.
\begin{lemm}  \label{spans} $h^0(\mathcal I_{F_h}(1)) = 0$ that is $F_h$ spans $\mathbb P^5$. \end{lemm}
\begin{proof} We have $\langle F \rangle = \mathbb P(A)$ since $F = F' \cup h$ is complete. Moreover $\overline r_{\vert F'}: F' \to \mathbb P^5$ is an embedding. Assume
$h^0(\mathcal I_{F_h}(1)) \geq 1$, then $F_h$ is contained in a hyperplane $L$. But then  the pull-back of $L$ by $\overline r_{\vert _{\mathbb P(A)}}$ contains $\langle F \rangle$: a contradiction.
\end{proof}
Now we assume $W = H^0(\mathcal I_C(2))$ for our usual vector space $W$ and that the inclusion of $\mathbb H = \vert \mathcal I_Y(2) \vert$ in $\mathbb P(W)$ is induced by the standard exact sequence of global sections 
$$
0 \to H^0(\mathcal I_{Y_h}(2)) \to H^0(\mathcal I_C(2)) \to H^0(\mathcal O_{Y_h}) \to 0.
$$
As already remarked  this is not restrictive up to projective equivalence. As in the proof of lemma \ref{crucial} let $\mathcal S = \Sym^2 T_{\mathbb P^5}(-1)$. Then $\mathcal S^*_o \subset H^0(\mathcal O_{\mathbb P^5}(2))$ is the space of quadratic forms singular at $o \in \mathbb P^5$ and this inclusion defines a monomorphism
$$
\upsilon: \mathcal S^* \to H^0(\mathcal O_{\mathbb P^5}(2)) \otimes \mathcal O_{\mathbb P^5}.
$$
Restricting $\upsilon$ to $\Sing C$ we then construct the Cartesian square
\begin{equation}
\begin{CD}
\mathcal N @>>> W \otimes \mathcal O_{\Sing C} \\
@VVV @VVV \\
{\mathcal S^*_{\Sing C}} @>>> {H^0(\mathcal O_{\mathbb P^5}(2)) \otimes \mathcal O_{\Sing C}}. \\
\end{CD}
\end{equation}
$\mathcal N$ is a rank $3$ vector bundle over the finite scheme $\Sing C$. Indeed, we have
\begin{equation}
\mathcal N_o = H^0(\mathcal I_C(2)) \cap H^0(\mathcal I^2_o(2)) 
\end{equation}
and we know that $L_o := H^0(\mathcal I_{Y_h}(2)) \cap H^0(\mathcal I^2_o(2))$ has dimension $2$. Since $C$ is a quadratic section of $Y_h$ singular at $o$, the above exact sequence
implies $\dim \mathcal N_o = 3$. Let $N_o := \mathbb P(\mathcal N_o)$, then $N_o$ is the net of quadrics through $C$ singular at $o$.
In particular it is clear that the map associated to $\mathcal N$ is the embedding sending $o$ to $N_o$, say
\begin{equation}
f_N: \Sing C \to \mathbb G.
\end{equation}
Now $\overline r: \mathbb G \to \mathbb G^*_{\mathbb H}$ associates to $N_o$ the pencil $\mathbb P(L_o)$. Moreover $Y_h$ is the locus in $\mathbb G^*_{\mathbb H}$ of the pencils of quadrics $\mathbb P(L_o)$, singular at $o \in Y_h$ and containing $Y_h$, see \ref{H_Y} and remark \ref{ triangle}. Hence it follows $\overline r \circ f_N = id_{\Sing C}$ and therefore we have
\begin{equation}
\label {id_Sing C} f_N(\Sing C) \subseteq F'.
\end{equation}
Since $F = h \cup F'$ is by assumption the Morin configuration defined by $A$, we have $\langle F \rangle = \mathbb P(A)$ and $F = \mathbb P(A) \cdot \mathbb G$. This implies that 
$F' = f_N(\Sing C)$. \medskip \par
After these remarks we can describe explicitly Morin-Del Pezzo configurations and construct an irreducible family which includes all these configurations. To this purpose we invert now the 
previous construction and start from a reduced $C \in \vert \mathcal O_{Y_h}(2) \vert$ such that $\Sing C$ spans $\mathbb P^5$. We define the embedding  $f_N: \Sing C \to \mathbb G$ as above and set: \begin{equation}  F' := f_N(\Sing C) \  \text{and} \  \mathbb P(A) := \langle h \cup F' \rangle. \end{equation}
\begin{thm} $A$ is maximal isotropic. \end{thm}
\begin{proof} Let $\ell$ be the length of $F'$. We show by induction on $0 \leq k \leq \ell$ that, for any subscheme $S' \subset F'$ of length $k$,  $\langle h \cup S' \rangle = \mathbb P(A_{S'})$
where $A_{S'}$ is isotropic. For $k = 0$ $h \cup S' = h$ and the statement is clearly true. Let $S'$ be of length $k+1$ then we observe that $S' $ is the biregular
image of $f_N(S)$, where $S$ is a subscheme of length $k$ of $\Sing C$. Let $o \in S$, then $f_N(o)$ is the parameter point of the net of quadrics $N_o$, containing $C$ and
singular at $o$. It is also clear that $\langle S' \rangle$ is spanned by the lines $\langle \zeta'_1 \rangle, \dots, \langle \zeta'_k \rangle$, where $\zeta'_i = f_N(\zeta_i)$ and  $\zeta_i$ denotes 
a subscheme of length $2$ of $S$ containing the point $o$. If $o_i \in \zeta_i$ and $o_i \neq o$ we denote by $n_i$ a vector defining the point $f_N(o_i)$. If $o = o_i$ then $n_i$ denotes a
non zero tangent vector to $\mathbb G$ at $f_N(o)$ defining the tangent line $\langle \zeta'_i \rangle$. Finally let $n$ be a vector defining $f_N(o)$ and let $s_1 \dots s_5$ be vectors 
respectively defining $h_1 \dots h_5$. Then $A_{S'}$ is generated by $s_1, \dots, s_5, n, n_1, \dots, n_k$. By induction $s_1, \dots, s_5, n_1, \dots, n_k$ generate an isotropic space. Moreover  $n$ is isotropic. Hence $A_S$ is isotropic if $$ w(n,s_i) = w(n,n_j) = 0 $$
 for $i = 1 \dots k$ and $j = 1 \dots 5$. Since the tangent space to $\mathbb G$ at
any point is isotropic, we have $w(n,n_j) = 0$ for every $n_j$ such that $\langle \zeta'_j \rangle$ is tangent to $\mathbb G$ at $o$. Otherwise we have $o \neq o_j$ and we are left to show that $w(n,n_j) = 0$.
To prove this we argue as follows, leaving some details to te reader. Let $N_o$ and $N_j$ respectively be the net of singular quadrics defined by $o$ and $o_j$ as above. To prove $w(n, n_j) = 0$ it suffices to show that $N_o \cap N_j$ is non empty. Let $\beta: Y_h \to \mathbb P^3$ be the projection from the line $b := \langle o o_j \rangle$. If $b$ is not in $Y_h$ then $\beta (Y_h)$ is an integral cubic surface. Moreover $\beta(C)$ is a $4$-nodal canonical curve. In particular it follows that $\beta(C) = \overline Q \cdot \beta(Y)$, where $\overline Q$ is a quadric surface. Let $Q = \beta^* \overline Q$, then $Q$ is a quadric of rank $4$, singular along the line $\langle o o_j \rangle$ and contains $C$. Hence we have $Q \in N_o \cap N_i$. Finally it is clear from section 5 that $w(n, s_i) = 0$. This shows by inudction  that $\langle h \cup F' \rangle$ is the projectivization of an isotropic space.   Since it is isotropic $A$ has dimension $\leq 10$. On the other hand the assumption that $\Sing C$ spans $\mathbb P^5$ implies that $r(A)$ is $6$-dimensional. Since $[h] \subseteq \Ker r \vert A$, it follows $\dim A = 10$. \end{proof}
 In what follows we will denote by $\mathcal C$ the family of curves like $C$, that is
 \begin{equation}
 \mathcal C := \lbrace C \in \vert \mathcal O_{Y_h}(2) \vert \ \vert \ h^0(\mathcal I_{\Sing C}(1)) = 0 \ \text {and $C$ is reduced} \rbrace.
 \end{equation}
 Notice that then $\Sing C$ has length $\geq 6$. Now let $\mathcal V$ be the family of all reduced curves $D \in \vert \mathcal O_{Y_h}(2) \vert$ such that $\Sing D$ has length $\geq 6$, it is known that
 $\mathcal V$ is integral, \cite{Te}. Moreover, it is easy to see that a general $D$ in the family is an integral nodal curve such that $\Sing D$ consists of six nodes in general position in $Y_h \subset \mathbb P^5$. In particular the conditions defining $\mathcal C$ are open and not empty on $\mathcal V$, so that $\mathcal C$ is integral. In a similar way we can define and use the universal singular point over $\mathcal C$, that is the family
\begin{equation}
\mathcal S := \lbrace (C,o) \in \mathcal \mathcal C \times Y_h \ \vert \ o \in \Sing C \rbrace.
\end{equation}
Fixing $o \in Y_h$, let $\mathcal S_o$ the fibre the projection $\mathcal S \to Y_h$ and let $\sigma: Y_o \to Y_h$ be the blowing up of $o$. Then $Y_o$ is a quartic Del Pezzo surface. Moreover,
the strict transform by $\sigma$ of the family of curves $\mathcal S_o$ is just an open set in the variety $\mathcal V'$ of all antibicanonical curves $C' \subset Y_o$ which are reduced and such that $\Sing C'$ has length $\geq 5$. Again $\mathcal V'$ is known to be integral of constant dimension $7$, \cite{Te}. Hence the next lemma follows.
\begin{lemm} $\mathcal S$ and $\mathcal C$ are integral. \end{lemm}
Now, to globalize slightly, we fix our notation as follows. Let $C \in \mathcal C$, then we set $W_C := H^0(\mathcal I_C(2))$ and consider the rank $6$
vector bundle $ \pi: \mathcal W \to \mathcal C$, whose fibre at $C$ is $W_C$. Passing to wedge product, we have the Grassmann bundle
\begin{equation}
\mathcal G \subset \mathbb P(\wedge^3 \mathcal W) \stackrel{\wedge^3 \pi} \longrightarrow  \mathcal C,
\end{equation}
whose fibre $\mathcal G_C$ is the Grassmannian of planes of $\mathbb P(H^0(\mathcal I_C(2))$, and the $\mathbb P^9$-bundle
\begin{equation}
\mathcal P \subset  \mathbb P(\wedge^3 \mathcal W) \stackrel{\wedge^3 \pi} \longrightarrow  \mathcal C,
\end{equation}
whose fibre $\mathcal P_C$ is $\mathbb P(A_C)$ and $A_C \subset \wedge^3 W_C$ is the isotropic space $A$ as above. Let
\begin{defi} \it The universal Morin-Del Pezzo configuration over $\mathcal C$ is the closed set
\begin{equation}
\mathcal Z :=  \mathcal G \cap \mathcal P \subset \mathbb P(\wedge^3 \mathcal W).
\end{equation}
\end{defi}
Some comments now are due. Let $f: \mathcal S \to \mathcal P$ be the morphism defined by the assignment $(C,o) \to N_o$, where $N_o$ is a net of quadrics as above. It is clear that
\begin{equation}
\mathcal Z' := f(\mathcal S)
\end{equation}
is an irreducible component of $\mathcal Z$. $\mathcal Z$ contains as well the five irreducible components \begin{equation} \mathcal H_i := s_i(\mathcal C), \ i = 1 \dots 5, \end{equation} where $s_i: \mathcal C \to \mathcal P$ is the section such that $s_i(C) := h_i \in \mathbb H = \vert \mathcal I_{Y_h}(2) \vert \subset \mathcal P_C = \mathbb P(A_C)$. Let $\mathcal H := \cup \mathcal H_i$,
so far we have $\mathcal Z' \cup \mathcal H \subseteq \mathcal G \cap \mathcal P$. In the next theorem we show that the latter is an equality. Of course this implies that each fibre of the family
\begin{equation}
\wedge^3 \pi: \mathcal G \cap \mathcal P \longrightarrow \mathcal C
\end{equation}
is a finite and complete configuration of incident planes, in particular a Morin-Del Pezzo configuration. This completes our description of these configurations.
\begin{thm} $\mathcal G \cap \mathcal P = \mathcal H_1 \cup \dots \mathcal H_5 \cup \mathcal Z'$. \end{thm}
To prove the theorem we proceed at follows. Let $z \in \mathcal G \cap \mathcal P$ be a point in the fibre over $C \in \mathcal C$. Then $z$ is the parameter point of a net of quadrics
$N \subset \vert \mathcal I_C(2) \vert$ and we have to show that $z \in \mathcal Z' \cup \mathcal H$. If $z$ is in $h$ then there is nothing to show. Hence we can assume that $z$ is
not in $h$, in other words that $N$ is not in the hyperplane $\mathbb H$ of quadrics through $Y_h$. Then $L := N \cap \mathbb H$ is a pencil of quadrics singular at some point 
$v \in Y_h$. Its base scheme is a cone $B_v$ of vertex $v$ over an integral complete intersection of two quadrics in $\mathbb P^4$, see \ref{Y_u}. Hence the base scheme of $N$ is  
 \begin{equation} X =  Q \cdot B_v \end{equation}
where $Q \in N - L$. In particular it is clear that  $X \cdot Y = Q \cdot Y = C$.
 \begin{lemm} $\Sing C \subset \Sing X$. \end{lemm}
\begin{proof} Let $o \in \Sing C$, we can assume $o \neq v$. Since $N$ defines a point of $\mathbb P(A)$ and $A$ is isotropic we have $N \cap N_o \neq \emptyset$. Hence there exists a quadric 
$Q_o \in N$ which is singular at $o$. If $Q_o$ is not in $\mathbb H$ then $X = Q_o \cdot B_v$ and $o \in \Sing X$.  If $Q_o$ is in $\mathbb H$ then $Q_o \in B_v$. In this case 
$\Sing Q_o$ contains $z, v$ and the line $E := \langle ov \rangle$. Let $\pi_E: Y \to \mathbb P^3$ be the projection from $E$ then $\pi_E(Y)$ is a quadric. Moreover, it is easy to deduce that then $E \subset Y$ and that the cone $B_v$ is singular along $E$. Hence we have $o \in X \cap \Sing B_v \subset  \Sing X$.  \end{proof}
\begin{lemm} Let $v$ be as above then $v \in \Sing C$. \end{lemm}
\begin{proof} Let $X= Q \cdot B_v$ and let $q_v \in H^0(\mathcal O_{\mathbb P^5}(1))$ be the polar form of $v$ with respect to $Q$. If $q_v = 0$ then $Q$ is singular at $v$ and the statement follows.
If $q_v(v) \neq 0$ then $v$ is not in $X$ nor in $Q$. In this case consider the projection $\pi_v: X \to \mathbb P^4$ from the vertex $v$ of $B_v$. Then $\pi_v$ is a finite double covering of $\pi_v(X)$, which is an integral complete intersection of two quadrics. Since $X = Q \cdot B_v$ the ramification divisor of $\pi_v$ is the hyperplane section of $X$ by $\lbrace q_v = 0 \rbrace$.  In particular $q_v$ vanishes on $\Sing X$. But then, by the previous lemma, $q_v$ vanishes on $\Sing C$. Since we are assuming $h^0(\mathcal I_{\Sing C}(1)) = 0$, we have a contradiction.  
If $q_v(v) = 0$ and $q_v \neq 0$ assume $v \notin \Sing C$ and observe that the line $\langle v p \rangle$ is in $Q$ for each $p \in \Sing C$. Indeed, $\langle v p \rangle$ is tangent to $Q$ at $v$ and contains $p$. Then $q_v$ vanishes in $\Sing C$ and the same contradiction follows. Hence $v \in \Sing C$. \end{proof}
The lemma implies that the net $N$, corresponding to $z \in \mathcal G \cap \mathcal P$, is the net $N_v$ of all quadrics through $C$ singular at $v$. Hence $z \in \mathcal Z'$ and
the proof of the theorem follows.
\begin{rem} \label {endrem} We  point out that, as a consequence of our description, a general Morin-Del Pezzo configuration is obtained from a nodal, integral canonical curve $C \subset Y_h$ with
exactly $6$ nodes. Notice also that $C \subset \mathbb P^1 \times \mathbb P^2$ so that its projection in $\mathbb P^2$ is a nodal sextic with $10$ nodes.
\end{rem}
\section{Morin configurations of higher length via canonical curves}
Finally we apply the previous results and constructions to deduce the uniqueness, up to projective equivalence in $\mathbb P^5$, of the finite Morin configuration having maximal cardinality $20$. 
We also outline the simple description of those families of configurations of length $k \geq 16$ having the one of maximal cardinality as a limit. We rely as previously on stable, highly singular canonical curves of genus $6$. \par
Let $F = \mathbb P(A) \cdot \mathbb G$ be a finite Morin configuration of planes in $\mathbb P^5$ of length $k \geq 16$. By theorem \ref{no plane} the $V$-threefold $V_A$ of $A$ contains a plane, say $P_h = \lbrace o \rbrace \times \mathbb P^2$ as in \ref{plane}. By proposition \ref{no fixed} $b := P_h \cdot \Sing V_A$ is the base scheme of a pencil of conics and it is finite.  Assume that $F$ has maximal cardinality $20$, then $F$ is smooth since its length is bound by $20$. Since $F - \lbrace u \rbrace$ is biregular to $\Sing V_A$ and $\Sing V_A$
 contains $b$, it follows that $b$ is a smooth complete intersection of two conics. Now we know from section $7$ that $V_A$ is the birational image of the product  map considered in \ref{reconstruct}, namely  
$$
\pi_1 \times \pi_2: \mathbb P^1 \times \mathbb P^2 \to \mathbb P^2 \times \mathbb P^2.
$$
We recall its definition.  We have $ C \subset Y \subset \mathbb P^1 \times \mathbb P^2 \subset \mathbb P^5$, where $Y$ is a smooth quintic Del Pezzo surface and $C \in \vert \mathcal O_Y(2) \vert$ is a canonical curve. Then $\pi_1: \mathbb P^1 \times \mathbb P^2 \to \mathbb P^2$ is defined by the net of divisors $\vert \mathcal J(2,2) \vert$, where $\mathcal J$ is the ideal sheaf of $C$, and the map $\pi_2: \mathbb P^1 \times \mathbb P^2 \to \mathbb P^2$ is the obvious projection. Moreover, $F$ is a Morin-Del Pezzo configuration and we have shown so far that $\Sing C$ is biregular to $\Sing (V_A - P_h)$. \par Therefore a smooth $F$ of cardinality $20$ is defined, up to $\Aut \mathbb P^5$, by a nodal curve $C \subset Y$ such that $\vert \Sing C \vert = 15$. Finally it is well known, and easy to see, that the unique curve $C$ such that $\Sing C$ is smooth of cardinality $15$ is the curve $C_{\ell}$ union of the $10$ lines of $Y$.  This proves the next uniqueness theorem.
\begin{thm} Up to $\Aut \mathbb P^5$ a unique finite Morin configuration of $20$ planes exists and it is the Morin-Del Pezzo configuration defined by the curve $C_{\ell}$. \end{thm}
Notice that $C_{\ell}$ is invariant under the action of $\Aut Y$, which is the symmetric group $\frak S_5$. Actually $C_{\ell}$ is a stable graph curve which is uniquely defined by its associated  graph
$\Gamma$. This has $10$ vertices corresponding to the $10$ lines of $C_{\ell}$. Each edge of $\Gamma$ corresponds to a node $o \in \Sing C_{\ell}$ and joins the vertices corresponding to the two lines through $o$. In our situation $\Gamma$ is the  famous
\it Petersen graph \rm $\Gamma$.   
 \par We do not address a systematic study of the stratification by their length of Morin-Del Pezzo configurations. We simply outline here some simple ways of smoothing partially $C_{\ell}$ so to obtain
some of the missed families of length $k \in [16,19]$. To this purpose just consider suitable connected subgraphs $\lambda$ of arithmetic genus zero and consider the family of graph curves 
defined by the graph $\Gamma_{\lambda}$, obtained from $\Gamma$ after contracting $\lambda$ to a point. Let $L \subset C_{\ell}$ be the curve defined by $\lambda$, that is $ L = L_1 + \dots + L_n$
where the summands correspond to the vertices of $\lambda$. Then the linear system $\vert L \vert$ is very ample. We have $L^2 = n-2$ and $C_{\ell}D = 2n$. Let $D \in \vert L \vert$ be general and 
\begin{equation}
C = C_{\ell} - L + D.
\end{equation}
It is easy to see that $C$ is nodal and that $\vert \Sing C \vert = 15 - n$. Moreover for $1 \leq n \leq 4$ the construction provides a curve $C$ such that $\Sing C$ spans $\mathbb P^5$. Let
$\mathbb P^5_C$ be $5$-space of quadrics through $C$, then $\Sing C$ defines in it, as usual, a Morin-Del Pezzo configuration $20 - n$ planes. Iterating the contraction to a point of a genus $0$ subgraph, one can describe all the irreducible families of Morin configurations of length $k \geq 16$ and their quotients by $\Aut Y$. Hopefully this matter will be reconsidered elsewhere. \par
 \medskip \par  \underline{\it Concluding remarks} \par Some constructions in this paper, involving singular canonical curves of genus $6$, admit natural extensions to higher genus. Indeed let $W_g$ be a vector space whose dimension is the triangular number $\binom{g-2}2$. We can assume that $W_g$ is the dual of the space of quadratic forms vanishing on
a nodal canonical curve \begin{equation} C \subset \mathbb P^{g-1}. \end{equation}
Then the equality considered by Zak in \cite{A} has, as a special case, the following one
\begin{equation}
(g-3) +\binom{g-3}2 = \binom{g-2}2
\end{equation}
and this makes interesting to consider Morin configurations of $(g-4)$-spaces in the projective space $\mathbb P(W_g)$. Let $F$ be a finite Morin configuration in the Grassmannian
$\mathbb G$ of $(g-4)$-spaces of $\mathbb P(W_g)$. Among many other questions it is natural to ask: \medskip \par
\centerline{ \it What one can say about the maximal length of $F$? }
 \medskip \par
 Stable canonical curves $C$ of genus $g$ with many nodes provide interesting examples of finite families of incident $(g-4)$-spaces. Indeed let $\mathbb I_C$ be the linear system of quadrics through
a stable $C$ and let $\mathbb I_z := \lbrace Q \in \mathbb I \ \vert \ z \in \Sing Q, z \in \Sing C \rbrace$. It turns out that the orthogonal $P_z \subset \mathbb P(W_g)$ is a subspace of dimension $g-4$. Then the family
 $$
 F_C := \lbrace P_z, \ z \in \Sing C \rbrace
 $$
 is an example of family of incident $g-4$-spaces. Indeed  let $z_1, z_2 \in \Sing C$ be distinct points and let $P_{z_1}, P_{z_2} \subset \mathbb P(W_g)$ be the orthogonal
 $(g-4)$-spaces respectively of $\mathbb I_{z_1}, \mathbb I_{z_2}$. Then, with the same argument used in genus $6$, the codimension of the space spanned by $\mathbb I_{z_1} \cup \mathbb I_{z_2}$ turns out to
 be $\binom {g-4}2 = \dim W_{g-2}$. Equivalently $P_{z_1} \cap P_{z_2}$ is a point. Hence $F_C$ is a finite family of incident $(g-4)$-spaces of $\mathbb P^{\binom{g-2}2 -1}$. \par
 Now stable canonical curves $C$ which are union of lines are $3g-3$-nodal and provide smooth families $F_C$ of cardinality $3g-3$. Each curve $C$ of this type is uniquely defined by a suitable graph
 as in the case of the Petersen graph. For instance a \it generalized Petersen graph $G(2k-1,1)$ \rm, see \cite{FGW},  uniquely defines a stable canonical curve  
 $$ C_{g, \ell} \subset \mathbb P^1 \times \mathbb P^{k-1} \subset \mathbb P^{g-1}
 $$
of even genus $g = 2k$. Omitting the discussion of the odd genus case and several details, this is readily constructed in $\mathbb P^1 \times \mathbb P^{k-1}$ as follows. In the Segre embedding $\mathbb P^1 \times \mathbb P^{k+1}$ consider the lines $L_i := \mathbb P^1 \times t_i, \ i = 1 \dots k+1$,  where $t := \lbrace t_1 \dots t_{k+1} \rbrace$ is a set of points in general position in $\mathbb P^{k-1}$. On the other hand  one can construct in $\mathbb P^{k-1}$ three nodal rational normal curves $R'_1, R'_2, R'_3$ which are union of lines, have no common component and  contain $t$ as a subset of smooth points. Then we can define the curve
 \begin{equation}
 C_{g,\ell} = R_1 \cup R_2 \cup R_3 \cup L_1 \cup \dots \cup L_{k+1},
 \end{equation}
where $R_j := u_j \times R'_j$, $j = 1,2,3$ and $u_j \in \mathbb P^1$.  Let $F_{g, \ell} := \lbrace P_z, \ z \in \Sing C_{g,\ell} \rbrace$ be the set of incident $(g-4)$-spaces defined by $\Sing C_{g,\ell}$. For $g \geq 8$ it is natural to ask
wether $F_{g, \ell}$ is a Morin configuration and has maximal cardinality. In any case the study of graph curves like $C_{g,\ell}$ seems to be interesting in order to study Morin configurations and their relations to the geometry of canonical curves.

\end{document}